\documentclass[12pt, a4paper, reqno]{amsart}

\usepackage[utf8]{inputenc}
\usepackage{color,fullpage}
\usepackage{mathrsfs}
\usepackage{graphicx}
\usepackage{faktor}
\usepackage{amssymb, multirow}
\usepackage{amsmath, float, fbb}
\usepackage{amsthm, csquotes, microtype, fullpage, textcomp, wrapfig}
\usepackage[colorlinks=true, linkcolor=blue, citecolor=blue, urlcolor=blue, breaklinks=true]{hyperref}
\usepackage[capitalise]{cleveref}
\usepackage[all,cmtip]{xy}
\usepackage{float}
\usepackage{tikz}
\usepackage{cite}
\usepackage{caption} 
\usetikzlibrary{arrows,shapes,automata,backgrounds,petri,decorations}
\usetikzlibrary{automata} 
\usetikzlibrary[automata] 
\usepackage{graphicx} 
\usetikzlibrary{external,automata,trees,positioning,shadows,arrows,shapes.geometric}
\usepackage{tikz-cd}
\usetikzlibrary{decorations.pathreplacing,decorations.markings}
\tikzset{close/.style={near start,outer sep=-2pt}} 
\tikzset{
  on each segment/.style={
    decorate,
    decoration={
      show path construction,
      moveto code={},
      lineto code={
        \path [#1]
        (\tikzinputsegmentfirst) -- (\tikzinputsegmentlast);
      },
      curveto code={
        \path [#1] (\tikzinputsegmentfirst)
        .. controls
        (\tikzinputsegmentsupporta) and (\tikzinputsegmentsupportb)
        ..
        (\tikzinputsegmentlast);
      },
      closepath code={
        \path [#1]
        (\tikzinputsegmentfirst) -- (\tikzinputsegmentlast);
      },
    },
  },
  mid arrow/.style={postaction={decorate,decoration={
        markings,
        mark=at position .5 with {\arrow[#1]{stealth}}
      }}},
}

\usepackage{mathtools}
\usepackage{hyperref,url}
\newtheorem{thm}{Theorem}[section]
\newtheorem{cor}[thm]{Corollary}
\newtheorem{lem}[thm]{Lemma}
\newtheorem{prop}[thm]{Proposition}

\newtheorem{conj}[thm]{Conjecture}
\theoremstyle{definition}
\newtheorem{defn}[thm]{Definition}

\theoremstyle{definition}

\newtheorem{expi}[thm]{Experiment}

\numberwithin{equation}{section}

\newcommand{\Q}{\mathbb{Q}}

\newcommand{\Z}{\mathbb{Z}}
\newcommand{\F}{\mathbb{F}}

\newcommand{\Sel}{\operatorname{Sel}}

\DeclareFontFamily{U}{wncy}{}
    \DeclareFontShape{U}{wncy}{m}{n}{<->wncyr10}{}
    \DeclareSymbolFont{mcy}{U}{wncy}{m}{n}
    \DeclareMathSymbol{\Sh}{\mathord}{mcy}{"58}

\begin{document}

\title{Unveiling Arithmetic Statistics of Congruent Number Elliptic Curves via Data Science and Machine Learning}

\author{Priyavrat Deshpande}
\address{Chennai Mathematical Institute, India.}
\email{pdeshpande@cmi.ac.in}

\author{Aditya Karnataki}
\address{Chennai Mathematical Institute, India.} \email{adityack@cmi.ac.in}

\author{Pratiksha Shingavekar}
\address{Chennai Mathematical Institute, India.}
\email{pshingavekar@gmail.com}
\thanks{The authors are partially supported by a grant from the Infosys Foundation}

\begin{abstract}
This article presents a comprehensive data-scientific investigation into the arithmetic statistics of congruent number elliptic curves, leveraging a dataset of square-free integers up to $3$ million. 
We analyze the Mordell-Weil ranks, 2-Selmer ranks, and 3-Selmer ranks of the corresponding elliptic curves $E_D: y^2 = x^3 - D^2x$, where $D$ is a square-free number. 
Our study empirically examines the Heath-Brown heuristics, which predict the distribution of $2$-Selmer ranks as well as congruent numbers based on their residue modulo $8$. 
In particular, offering statistical insights into the proportion of numbers whose associated elliptic curves have positive rank. 
We provide a rigorous verification of Goldfeld's Conjecture in this context, analyzing the distribution of analytic ranks and demonstrating their alignment with the conjectured $50/50$ split for ranks $0$ and $1$. 
Furthermore, we explore the conjectural asymptotic distribution of $2-$ and $3$-torsion part of the Tate-Shafarevich group of these curves. 
Based on empirical evidence, we also suggest potential statistical distribution of $3$-Selmer and Mordell-Weil ranks. 
We also examine the averages of Frobenius traces and observe that they tend to zero without exhibiting any murmuration-like patterns. 

In addition to these number-theoretic analyses, we apply machine learning techniques to classify and predict congruent numbers, exploring the efficacy of computational methods in distinguishing congruent from non-congruent numbers based on the arithmetic properties of elliptic curves. 
This interdisciplinary approach blends advanced number theory with modern data science, providing empirical support for conjectures as well as discovery of new patterns. 
\end{abstract}

\maketitle

\section{Introduction}
The study of elliptic curves is at the forefront of modern mathematics, bridging diverse areas such as algebra, geometry, number theory, complex analysis and cryptography.  
A classical and particularly intriguing problem within this field is the congruent number problem, which asks for which positive integers $D$ can be the area of a right-angled triangle with all three sides of rational length. 
This seemingly elementary question, dating back to ancient times, finds its elegant modern formulation in the theory of elliptic curves: a positive integer $D$ is congruent if and only if the associated elliptic curve, given by the equation $y^2 = x^3 - D^2x$ has a rational point of infinite order, equivalently its Mordell-Weil (MW) rank is nonzero. 

Understanding the distribution of these ranks across families of elliptic curves is a central challenge and the congruent number family is no exception. 
However, two profound works guide much of the research related to ranks of congruent number curves. 
The first is Heath-Brown's heuristics, which offer specific predictions for the density of congruent numbers based on their residue modulo 8. 
However, directly computing the MW rank of an elliptic curve is computationally intensive and often intractable for large numbers (in the congruent number context). 
Consequently, Heath-Brown's analysis relied on $2$-descents, particularly the $2$-Selmer groups, whose dimension (as an $\mathbb{F}_2$-vector space) provides a computable upper bound for the MW rank.
In particular, for $D\equiv 1, 2, 3 \pmod 8$, the average dimension of the $2$-Selmer group tends to be even. 
Combined with the parity conjecture it means that the MW rank in this case is also even; often $0$.
This means that such numbers are often not congruent. 
On the other hand for $D\equiv 5, 6, 7 \pmod 8$ the $2$-Selmer dimension is odd, implying, via the parity conjecture, that such numbers are expected to be congruent, with MW rank $1$ or higher. 
Further, Heath-Brown's work provides a unified theoretical framework to explain the statistical distributions of $2$-Selmer groups by making precise predictions about moments of sizes of these groups. For example, he proved that the average size of $2$-Selmer groups is $3$. 

The second is the Goldfeld's Conjecture, which, in its general form, posits a remarkable statistical regularity for the \emph{analytic ranks} of quadratic twists of a fixed elliptic curve: it predicts that approximately $50\%$ of such twists will have rank $0$ and, $50\%$ will have rank 1, and a vanishingly small proportion will have rank 2 or higher.
Goldfeld's conjecture is applicable in this context, since congruent number curves are quadratic twists of $y^2 = x^3 - x$.

This article presents a comprehensive data-scientific study of congruent number elliptic curves, utilizing a substantial dataset comprising square-free integers up to $3$ million. 
For each such $D$, we have collected data on its $2$-Selmer rank and the MW rank. 
Moreover, we have also computed, $3$-Selmer rank, analytic rank, Frobenius traces, regulator, real period and other related numerical parameters for a lot of the curves corresponding to square-free numbers up to $1$-million. 
Our empirical approach allows for a direct statistical confrontation with theoretical predictions. 
Specifically, we undertake several key experiments:

\begin{enumerate}
    \item An empirical examination of Heath-Brown's heuristics \cite{hb, hb1}, analyzing the observed distribution of $2$-Selmer ranks across different residue classes modulo 8.
    \item Empirical verification of Delaunay's heuristics \cite{del, del1} regarding torsion part of the Tate-Shafarevich groups. 
    \item A detailed verification of Goldfeld's Conjecture \cite{gold} within the context of congruent number elliptic curves, focusing on the observed distribution of analytic ranks.
    \item A detailed analysis of the recent work of Alex Smith \cite{smith2016congruent} regarding the density of congruent numbers and the validity of Birch and Swinerton-Dyer conjecture in the context of this family of elliptic curves. 
    \item Motivated by the murmurations phenomenon \cite{mumur01} we also experiment with average of Frobenius traces taken over congruent number curves. 
    \item An exploration of the distribution of $3$-Selmer ranks and MW ranks of congruent number elliptic curves, providing novel insights into their statistical behavior. 
    \item The application of machine learning techniques to classify and predict congruent numbers, demonstrating the utility of computational intelligence in number theory.
\end{enumerate}

We present several of our principal findings. A significant result by Heath-Brown involves providing a theoretical estimate of the moments of $2$-Selmer sizes. 
He demonstrates that these moments are independent of the residue class of $D$ modulo $8$, with an associated error term involving powers of $\log X$ and $\log(\log X))$ (\Cref{hb2thm1}). 
Our empirical verification suggests that, for numbers of the form $1, 3 \pmod{8}$, the observed moments deviate from the theoretical predictions; however, for the remaining two residue classes, the observed moments closely correspond with the predicted values. 
Notably, for $1\pmod{8}$ the empirical values exceed, whereas for $3\pmod{8}$, they fall short (\Cref{momentslog}).

Additionally, we observe that in the case of the average size of the $2$-Selmer group, the error term is mostly precise, meaning that the magnitude is accurate for $1, 3\pmod{8}$ residue classes, but deviates for the other two residue classes. 
For the $2$nd moment, the bound on the error term does not appear as precise, indicating potential improvements in theoretical estimations may be feasible (\Cref{loglogterm}).

Subsequent experimentation aimed at determining the average value of Frobenius traces yielded intriguing results. 
Specifically, for a fixed prime, the average Frobenius trace value across all elliptic curves in our dataset tends towards zero, resulting in an outcome (\Cref{frobavg}) not, to our knowledge, discussed in existing literature. 
Motivated by this, we extended our research to cubic and quartic twists, observing a consistent phenomenon: for a fixed prime, the average Frobenius trace tends towards zero. 
Consequently, we propose a conjecture that this behavior is applicable to any twist family (\Cref{openq}).

Another notable observation in our work pertains to the empirical behavior of $3$-Selmer groups. 
Firstly, the parity distribution mirrors that of $2$-Selmer's. 
Secondly, the average size remains largely constant across all residue classes, aligning closely with the predicted average of $4$ for all elliptic curves. 
Thirdly, the empirical probability distribution of ranks aligns closely with the conjecture by Poonen-Rains. 
Furthermore, these phenomena become apparent without the necessity of extensive datasets.

This work is an emerging trend in modern mathematics: the application of data science and computational techniques to explore and generate hypotheses in traditionally theoretical domains like number theory. 
In a nutshell, the art of mathematical data science can be described as follows: one starts by considering certain mathematical objects collectively to generate datasets.  
The next step involves applying machine learning tools to find statistical structures in the datasets. 
The final step consists of interpreting these results to understand the mathematical objects better by means of finding newer structures, conjectures and theorems. 
This interdisciplinary synergy promises to accelerate discovery, providing empirical evidence to complement and inform theoretical advancements. 
We should note that, these processes are not entirely new to mathematics, in fact, the formulation of Birch and Swinerton-Dyer conjecture is credited to computer generation of elliptic curves and further analysis of the resulting data. 
Past couple of years have seen increasing interest in using large datasets, statistical analysis, and machine learning to uncover patterns, test conjectures, and even guide proofs in areas such as the distribution of prime numbers, properties of $L$-functions, arithmetic of elliptic curves, low-dimensional topology, algebraic combinatorics and algebraic geometry to name a few.
We refer the interested reader to two excellent surveys \cite{davies2021advancing, douglee25} which summarize important developments in this field. 
The following two papers warrant a special mention since they are relevant to our work: in the paper \emph{Murmurations of elliptic curves} \cite{mumur01}, the authors, among other things, find a murmuration-like pattern in the averages of Frobenius traces. 
This is an example of a purely data-driven insight. 
The authors also train various machine learning models to predict the MW ranks. 
We should also mention \cite{shaorder}, where authors train neural networks, in addition to ML models, to predict the order of (analytic) Tate-Shafarevich groups. 

The remainder of this paper is organized as follows: Section \ref{prelims} starts with number theoretic preliminaries and also provides details of dataset generation and methodologies employed. 
\Cref{seldist} is about experimental analysis of $2$-Selmer groups of congruent number curves. 
In \Cref{heath-brown}, we first empirically analyze Heath-Brown's work. 
We recall each of his major theoretical heuristics results and then present our experimental findings. 
This is followed by an analysis of more general results of Poonen-Rains. 
Later in \Cref{delaunay} we crosscheck Delaunay's heuristics for $2-$ and $3$-Selmer groups. 
\Cref{gold-smith} discusses the empirical validation of Goldfeld's conjecture and the BSD conjecture in the context of congruent number curves. 
Next, in \Cref{sec:murmurations}, we move on to analyzing averages of Frobenius traces, along the lines of murmurations of elliptic curves.
\Cref{sec:mlexp} describes the machine learning experiments for congruent number classification and prediction. 
Finally, \Cref{conclude} outlines directions for future research based on new empirical patterns discovered in the context of $3$-Selmer groups, MW ranks and a visualization of parameters in the BSD formula. 

\section{Preliminaries}\label{prelims}
\subsection{The congruent number problem}\label{introCNP}
A positive integer $D$ is called a \emph{congruent number} if $D$ is the area of a right triangle with rational legs. 
Determining whether a given $D$ is congruent is known as the \emph{congruent number problem}. 
It suffices to consider only the square-free values $D$. 
Fibonacci proved that $5$ and $6$ are congruent, Euler showed that $7$ is a congruent number, while Fermat showed that $1$, $2$, $3$ are not. 
Although, Tunnell's criterion \cite{tunnell}, assuming the Birch and Swinnerton-Dyer (BSD) conjecture, gives a complete solution to the problem, an unconditional solution is still not known. 
The following results, among many other special cases, are known in this regard.
\begin{itemize}
    \item If $D=\ell \equiv 3 \pmod 8$ is a prime, then $D$ is non-congruent.
    \item If $D=\ell_1\ell_2$ with $(\ell_1,\ell_2) \equiv (1,3) \pmod 8$ and $\left(\frac{\ell_1}{\ell_2}\right) =-1$, then $D$ is non-congruent.
\end{itemize}

 We will later see, following the discussion in \Cref{introcurves}, that assuming the BSD conjecture, $D$ is a congruent number if $D \equiv 5,6,7 \pmod 8$. 

\subsection{The congruent number elliptic curves} \label{introcurves}
For a given rational number $D \ne 0$, let $E_D$ be the elliptic curve $E_D:= y^2=x^3-D^2 x$ defined over $\Q$. 
Then $E_D$ is a quadratic twist of $E_1:y^2=x^3-x$ by $D$. 
Since $E_D$ is isomorphic to $E_{c^2D}$ for any rational $c \ne 0$, we can assume that $D>1$ is a square-free integer. Further, via the change of coordinates $(x,y) \mapsto (-x,y)$, one can see that $E_D$ is isomorphic to $E_{-D}$ over $\Q$. 
Hence, one can further restrict one's attention to the cases $D >0$. It is a well-established fact that $D$ is congruent if and only if the rank $r(D)$ of the elliptic curve $E_D$ over $\Q$ is positive. 
We will now state some known facts from \cite[Ch. 10]{sil}. 
First, the torsion subgroup $E_D(\Q)_{tors}$ and the rank $r(D)$ of $E_D(\Q)$ satisfy 
\[E_D(\Q)_{tors} \cong \Z/2\Z \times \Z/2\Z \ \text{ and } \ r(D) \le 2 \nu(2D)-1,\] 
where $\nu(2D)$ is the number of distinct prime factors of $2D$.

Let $p$ be a prime not dividing $2D$ and $\#E_D(\F_p)$ be the number of solutions to $y^2 \equiv x^3-D^2x \pmod p$. Then we know that $\#E_D(\F_p) = p+1-a_p$ for some integer $a_p=a_p(E_D)$, called the Frobenius trace of $E_D$ modulo $p$. The $L$-function associated to $E_D$ is defined by

\[L_D(s):=\prod_{p \nmid 2D} (1-a_pp^{-s}+p^{1-2s})^{-1}.\]

Then $L_D(s)$ has an analytic continuation to the entire complex plane. 
The BSD conjecture then tells us that the order of vanishing $\text{ord}_{s=1} L_D(s)$, the so-called analytic rank $r_{an}(D)$ of $E_D$, is equal to $r(D)$. 
Due to the works of Gross-Zagier, Coates-Wiles and Rubin, we know that $r(D)=r_{an}(D)$ for the curves $E_D$ whenever $r_{an}(D)\le 1$. 
However, very little is known when $r_{an}(D) \ge 2$.  There is a functional equation relating $L_D(s)$ and $L_D(2-s)$ via a change of sign $\omega_D = \pm 1$, known as the root number of $E_D/\Q$. Using which, one can deduce that $(-1)^{r_{an}(D)} = \omega_D$. Birch and Stephens \cite{bs} showed that $\omega_D = +1$ when $D \equiv 1,2,3 \pmod 8$, while $\omega_D = -1$ when $D \equiv 5,6,7 \pmod 8$. In particular, $r_{an}(D)$ is positive for $D \equiv 5,6,7 \pmod 8$, and in these cases, $r(D)$ is positive if we assume the BSD conjecture be true.

Given a $D$, the conductor $N_D$ of $E_D$ is $16D^2$ if $D$ is even and $32D^2$ otherwise. We know that there exists a surjective morphism $\phi_D:X_0(N_D) \to E_D$ defined over $\Q$. The degree of $\phi_D$ is called the modular degree of $E_D$, denoted $m_D$. A conjecture by Watkins then predicts that $2^{r(D)} \mid m_D$. Thus, in particular, if $m_D$ is odd, then one can expect $r(D)$ to be $0$.

 Let $E$ be an elliptic curve over $\Q$ and $n$ be an integer $\ge 2$. For a prime $\omega \le \infty$ of $\Q$, denote by $\Q_\omega$ the completion of $\Q$ at $\omega$. Further, for a field $F \in \{\Q, \Q_\omega\}$, take $G_F$ to denote the absolute Galois group of $F$ and $\delta_{F}:E(F)/nE(F) \to H^1(G_F, E[n])$ to be the Kummer map. Taking the cohomology of
 \[0 \to E[n] \to E \to E \to 0,\]
 we get the following commutative diagram:
 \begin{center}\begin{tikzcd}
 0 \arrow[r] & {E}(\Q)/nE(\Q) \arrow[r, "{\delta}"] \arrow[d]
& H^1(G_\Q, E[n]) \arrow[d, "\underset{\omega}{\prod} {\rm res}_\omega"] \arrow[r] & H^1(G_\Q,E)[n] \arrow[d] \arrow[r] & 0\\
 0 \arrow[r] & \underset{\omega}{\prod} {E}(\Q_\omega)/nE(\Q_\omega) \arrow[r, "\underset{\omega}{\prod} {\delta}_{\omega}"] & \underset{\omega}{\prod} H^1(G_{\Q_\omega},E[n]) \arrow[r] & \underset{\omega}{\prod} H^1(G_{\Q_\omega},E)[n] \arrow[r] & 0.
\end{tikzcd}\end{center} 

\begin{defn}\label{mainsel}
The $n$-Selmer group of $E$ over $\Q$, denoted $\mathrm{Sel}_n(E/\Q)$, is defined as
 $$\mathrm{Sel}_n(E/\Q)= \{\xi \in H^1(G_\Q,E[n]) \mid {\rm res}_\omega(\xi) \in \mathrm{Im} \ \delta_{\omega} \text{ for every } \omega\}.$$ 
\end{defn}

Setting $\Sh(E/\Q):=\text{Ker}\big( H^1(G_\Q,E) \to \underset{\omega}{\prod} H^1(G_{\Q_\omega},E) \big)$ to be the Tate-Shafarevich group of $E$ over $\Q$,  we get the fundamental exact sequence of $\Z/n\Z$-modules:
\begin{equation}\label{eq:selshaseq}
    0 \to E(\Q)/nE(\Q) \to \mathrm{Sel}_n(E/\Q) \to \Sh(E/\Q)[n] \to 0.
\end{equation}
To ease the notation, we will denote by $S_2(D)$ the $2$-Selmer group of $E_D/\Q$. 
Then the presence of full rational $2$-torsion on $E_D$ accounts for the fact that $S_2(D)$ must be of size at least $4$. 
We will write $\#S_2(D) = 2^{2+s(D)}$. For $n=2$, using the above exact sequence of the $\F_2$-modules, we get that dimension over $\F_2$ of $S_2(D)$ gives a natural upper bound on $r(D)$, namely $r(D) \le s(D)$. Therefore, one has
\[r(D) \le s(D) \le 2 \nu(2D).\]

\subsection{Generation of the datasets} 
We now explain how our datasets were generated and give a brief overview of our experiment strategies. 
The \emph{$L$-functions and modular forms database} \cite{lmfdb}, currently has complete data for elliptic curves up to the conductor size $500,000$. 
As a result, there are only $87$ congruent number elliptic curves listed in LMFDB, so, we could not rely on their main database. 
However, one of their auxiliary databases contains some information about congruent number curves up to $1,000,000$. 
This data was created by Randall L. Rathbun and for each $n$, it contains conductor, regulator, analytic order of Sha, root number, generators of the MW group, the MW rank and a few other parameters. 
Since, our experiments only required square-free numbers, we trimmed the original dataset accordingly and created our first dataset containing all the $607926$ square-free numbers less than $1$ million and their conductor, regulator, MW rank, and the analytic order of Sha. 
The reader should note that, for brevity, in the article sometimes we do not distinguish between a square-free number and the corresponding elliptic curve. 

We then expanded our database by calculating the $2$-Selmer rank of all the congruent number curves corresponding to square-free numbers less than $3$ million; their number is $1823773$. 
The $2$-Selmer ranks were calculated using Sage Math's \texttt{selmer\_rank()} function \cite[Elliptic curves over rational numbers]{sagemath}. 
These values are unconditional since the calculations do not assume any conjectures.

In order to calculate the MW rank, we used the following strategy: for each $n$ greater than $1$ million we used the \texttt{RankBounds} function of Magma to determine lower and upper bounds of the rank, see the documentation for \cite{magma}, in the section titled \emph{Mordell-Weil Groups and Descent Methods}.
As the numbers grow large neither Magma's \texttt{Rank} function nor Sage's \texttt{rank()} function are useful. 
This is because Magma's \texttt{rank()} function usually returns the lower bound if it fails, while Sage Math's rank function raises a runtime error exception. 
On the other hand, the \texttt{RankBounds} function is the only available tool that implements descents over $\mathbb{Q}$ up to $9$ and always returns values. 
For curves corresponding to smaller numbers (i.e., less than $1$ million) these bounds are unconditional; no conjectures are assumed. 
However, as we go beyond $1$ million it takes at least a couple hours, sometimes even more, to process just one curve. 
To speed up the calculation Magma recommends assuming GRH. 
Hence, we also consult $2$-Selmer rank: if it is $0$ then the MW rank is also zero.
On the other hand, if it is $1$ then the MW rank is also $1$; this is conditional on BSD which implies that the dimension of Sha $2$-torsion is even. 
We also consult the list of congruent numbers up to $10^7$ (available at \url{https://www.numbersaplenty.com/set/congruent_number/}); the numbers that do not appear in that list are not congruent and the MW rank of the corresponding curves is set to $0$. 
The congruent numbers that are $1,2, 3\pmod{8}$ such that the associated curves have $2$-Selmer rank $2$ also have MW rank $2$. 

This way we are left with numbers whose curves have $2$-Selmer rank at least $3$. 
In case of square-free numbers up to $3$ million, such numbers are around $90,000$. 
We calculate \texttt{RankBounds} for these curves. 
If the bounds match then it is the MW rank of the curve. 
Next, if the bounds do not match, then we compute the root number and use the parity conjecture to decide the rank. 
Even if that fails (i.e., both bounds are of the same parity) then we call Sage Math's \texttt{analytic\_rank\_upper\_bound()} function, which does rely on GRH, but always returns a value. 
This upper bound is then used to decide the MW rank. 
This strategy is similar to the one described in more details in \cite[Section 2.2]{bhkssw}, however, the explicit Magma techniques described there are now part of the \texttt{RankBounds} function. 
It is important for the reader to recognize that with an increase in the value of $D$, the computational time required to ascertain rank bounds correspondingly rises. As of now, the determination of these unknown MW ranks remains incomplete. Nonetheless, their quantity constitutes less than $1\%$ of the total number of curves under consideration, and thus they are unlikely to produce significant distortions in the results. In instances where MW ranks have not been resolved, we have provisionally recorded $-1$ as the MW rank within our data files. These files will be subject to periodic revisions as our computations undergo further advancement.

In order to compute the analytic rank of these curves we use Sage Math's \texttt{analytic\_rank()} function by allowing it implement all the algorithms like \emph{PARI, sympow, rubinstein, magma, zero\_sum}. 
If in case the numbers don't match then we either use the root number or the \texttt{analytic\_rank\_upper\_bound} function to decide the analytic rank. 

We also compute the $3$-Selmer rank using Magma's \texttt{ThreeSelmerGroup} function. 
In many instances, class group and unit group computations proved expensive; in such cases Magma required use of \texttt{SetClassGroupBounds("GRH")}. 
Hence our $3$-Selmer ranks are conditional on GRH. 
The numerical parameters like regulator, special value, Tamagawa product etc. that appear in the BSD formula were calculated using Sage Math. 
We also compute Frobenius traces for the first $1000$ primes using Sage Math's \texttt{aplist} function.

Most of our data generation, analysis and ML experiments were performed on a machine with Intel Xeon CPU E5-2630 0 running at 2.30GHz with $15$ cores. 

\subsection{Availability of the datasets}
The Python scripts we developed to carry out the experiments are available for download from the GitHub repository located at 

\begin{center}
{\href{https://github.com/pratiksha-skar/CN_curves}{https://github.com/pratiksha-skar/CN\_curves}}   
\end{center}

The generated datasets are available for download on Zenodo. 
The files are either in the csv format or the parquet format. 
For example, the file \texttt{three\_mil\_2selmer.parquet} contains $2$-Selmer ranks of congruent number curves corresponding to square-free numbers up to $3$ million. 
The curve $E_D$ is identified by the value $D$ and then there is another column which has the $2$-Selmer ranks. 
The Zenodo link is provided in our GitHub repository. 
The interested reader should be able to reproduce our results.

\section{Distribution of Selmer groups}\label{seldist}

\subsection{Heath-Brown heuristics}\label{heath-brown}
In this section, we recall the main results of \cite{hb} and \cite{hb1}. Given $D>0$, recall that $s(D) = \dim_{\F_2} S_2(D) - 2$. 
Before we begin the data scientific study of Heath-Brown heuristics, let us first look at the frequency distribution of the Selmer ranks. 
As stated earlier we computed $2$-Selmer ranks of $1,823,773$ congruent number curves. 
These curves correspond, precisely, to all the square-free numbers up to $3$ million. 
The ranks encountered ranged from $0$ to $6$; \Cref{sel_freq} shows their frequency distribution. 
As can be seen, nearly $800000$ curves, which constitutes roughly $44\%$, are of rank $1$. 
Whereas rank $0$ and rank $2$ curves are roughly $22\%$ each. 
Remaining are the higher rank curves, out of which rank $5$ and $6$ curves are negligible. 

\begin{figure}[ht]
  \centering
  \includegraphics[width=0.8\textwidth,clip]{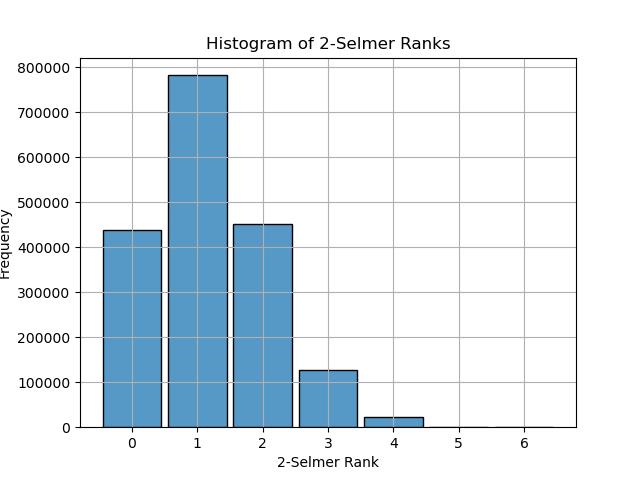}
  \caption{Frequency distribution of Selmer ranks}
  \label{sel_freq}
\end{figure}
We now look at the parity distribution of Selmer ranks across residue classes modulo $8$. The following statement first appeared in \cite{hb} and a proof by Monsky appeared in \cite[Appendix]{hb1}. 
\begin{thm}\label{selparity}
    If $D$ is positive and square-free, then $s(D)$ is even for $D \equiv 1, 2 \text{ or } 3 \pmod 8$, and odd for $D \equiv 5, 6 \text{ or } 7 \pmod 8$.
\end{thm}

The first column of \Cref{paritytable} has residue classes of $D$ modulo $8$, the second column has number of $D$'s in each residue class. The third and fourth column has the number of curves (in that residue class) whose $2$-Selmer groups have parity odd or even respectively. 

\begin{table}[ht]
  \centering
  \begin{tabular}{|c|c|c|c|c|c|}
    \hline
    Residue & Total curves & Odd parity & Even parity \\ \hline
    $1$ & $303961$ & $0$ & $303961$  \\ \hline
    $2$ & $303967$  & $0$ & $303967$  \\ \hline
    $3$ & $303961$  & $0$ & $303961$  \\ \hline
    $5$ & $303959$  & $303959$ & $0$  \\ \hline
    $6$ & $303962$ & $303962$ & $0$  \\ \hline
    $7$ & $303963$ & $303963$ & $0$  \\ \hline
  \end{tabular}
  \caption{The number of curves with given parity of their Selmer group}
  \label{paritytable}
\end{table}

The empirical data is clearly consistent with the Theorem \ref{selparity} above, as seen in \Cref{sel_distro}. 
We also look at a more refined aspect: the proportion of each specified $2$-Selmer rank in each residue class. 
\Cref{sel_distro} shows that for $D\equiv2\pmod{8}$, nearly $54\%$ curves have Selmer rank $0$, and in the case of $D\equiv5\pmod{8}$, nearly $85\%$ curves have $2$-Selmer rank $1$. 

\begin{figure}[ht]
  \centering
  \includegraphics[width=0.8\textwidth,clip]{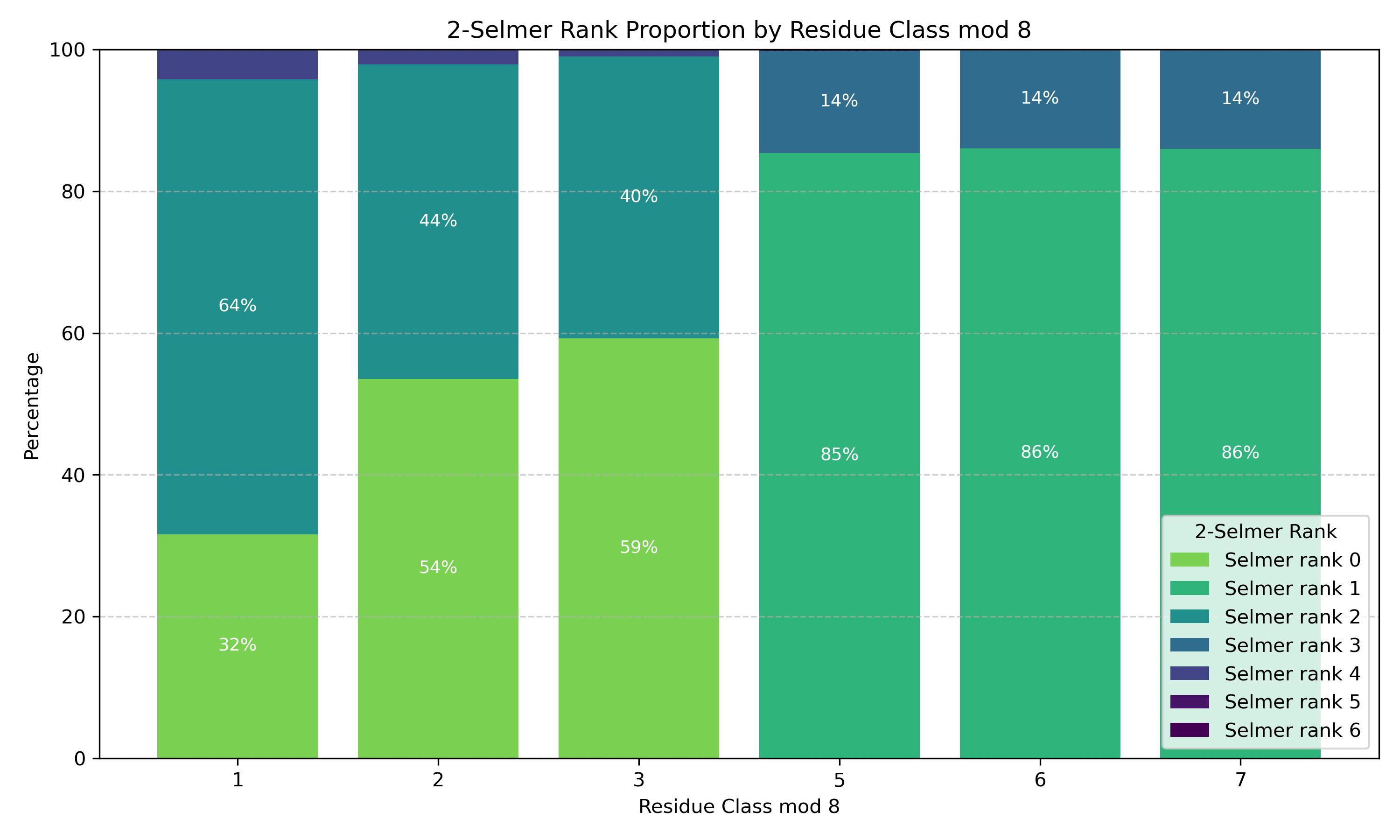}
  \caption{Proportion of Selmer ranks in each residue class}
  \label{sel_distro}
\end{figure}

Now we turn to congruent numbers. 
As stated earlier a (square-free) number is congruent if and only if the MW rank of the corresponding curve is nonzero. 
Since we have rank data for numbers up to slightly more than $1$ million, we only focus on those curves. 
The \Cref{cngpro} shows the histogram that displays the proportion of congruent numbers in each residue class. 
The empirical data reveals that slightly more than $10\%$ of square-free numbers which leave residue of $1$ or $3$ modulo $8$ are congruent.  
For positive integers that are $3\pmod{8}$ the percentage of congruent numbers $7.5$.  
On the other hand $100\%$ of the numbers that are $5,6,7\pmod{8}$ are congruent.  
This compels us to look at the distribution of the MW rank across these residue classes, however we postpone it till the end of this Section.

\begin{figure}[ht]
  \centering
  \includegraphics[width=0.8\textwidth,clip]{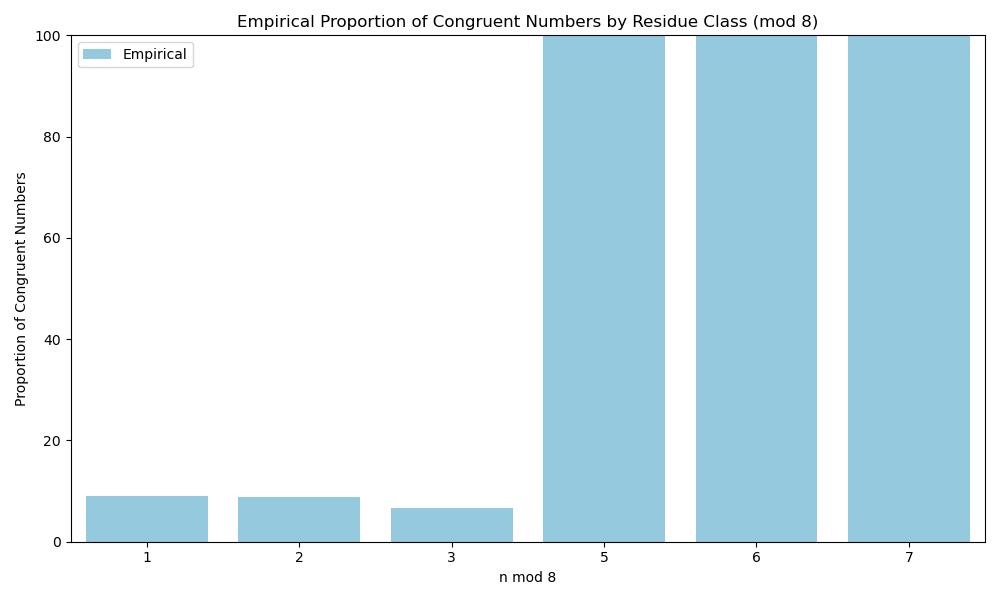}
  \caption{Proportion of congruent numbers in each residue class}
  \label{cngpro}
\end{figure}

If $E$ is an elliptic curve defined over a global field $F$, then it is conjectured that the average of $\#\mathrm{Sel}_p(E/F)$ is $p+1$. 
Bhargava and Shankar \cite{bs1} proved that the average size of $\mathrm{Sel}_2(E/\Q)$ is $3$, when one considers the family of all elliptic curves $E/\Q$. 
However, in the family of congruent number elliptic curves, this average is shifted by $2$, to account for the contribution by $E(\Q)[2]$. More precisely, one considers the average of $2^{s(D)}$, instead of $\#S_2(D)=2^{2+s(D)}$. 
\begin{defn}\label{defofS(X,h)}
    Let $h$ be an odd integer and $X>0$. Take $S(X,h)$ to be the set
\[S(X,h):=\{D \equiv h \pmod 8 \mid 0 < D \le X, \ D \text{ is square-free}\}.\]
\end{defn}

\begin{thm}{\cite[Theorem 1]{hb}}\label{avg2sel}
    For any odd integer $h$, we have 
    \[\sum_{D \in S(X,h)} 2^{s(D)} = 3\#S(X,h) + O(X(\log X)^{-1/4}(\log \log X)^{8}).\]
\end{thm}

To determine the distribution of $s(D)$ in more details, Heath-Brown \cite[Theorem 1]{hb1} also considers the average of higher moments of $s(D)$ to obtain the below
\begin{thm}\label{hb2thm1}
    For $h = 1, 3, 5$ or $7$, and any fixed positive integer $k$, we have \[\lim_{X\to\infty} \frac{1}{\#S(X,h)}\sum_{D \in S(X,h)} 2^{ks(D)} = c_k  + \frac{o_k(X)}{\#S(X,h)}, \]
where \[ c_k = \prod_{j=1}^k(1 + 2^j).\]
\end{thm}

There is a more general version of the above statement conjectured by Poonen and Rains, which says that for positive integers $m$, the average of $(\#\mathrm{Sel}_p(E))^m$ over all elliptic curves is $\prod_{i=1}^m(p^i + 1)$. 
The conjecture has been proved for $p=2, 3$. 

We now start presenting our findings regarding empirical verification of Heath-Brown's heuristics. 
First, we exhibit the most striking patterns predicted in Heath-Brown’s work — that the behavior of the $2$-Selmer rank depends statistically on the residue class mod $8$, with some classes favoring even parity and higher ranks.
However, at this point, due to \Cref{sel_distro} it should not come as a surprise. 
It is clear from \Cref{condtionalmoments}, that for $D\equiv 5, 6, 7\pmod{8}$, the average size is much closer to the theoretical prediction. 
In the case of $1\pmod{8}$ curves, the average size is bigger and for the remaining two the average size is smaller compared to the theoretical prediction. 
Hence, there is a lot of scope to further explore the error term.
\begin{figure}[ht]
    \centering
    \includegraphics[width=0.5\linewidth]{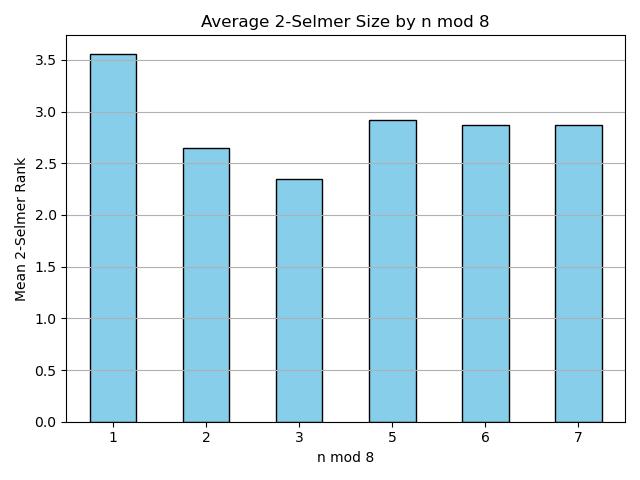}
    \caption{Average $2$-Selmer size per residue classes}
    \label{condtionalmoments}
\end{figure}

The first step in devising experiments to verify Heath-Brown's results was to reduce the size of the dataset by considering only odd square-free numbers. 
This brought down the number of curves to $1215844$. 
Then,  we compute the theoretical $c_k$ values, which are $3, 15, 35$ for $k = 1, 2, 3$ respectively. 
The empirical findings about the first $3$ moments of the Selmer size are described in \Cref{momentslog}

\begin{figure}[ht]
    \centering
    \includegraphics[width=0.6\linewidth,clip]{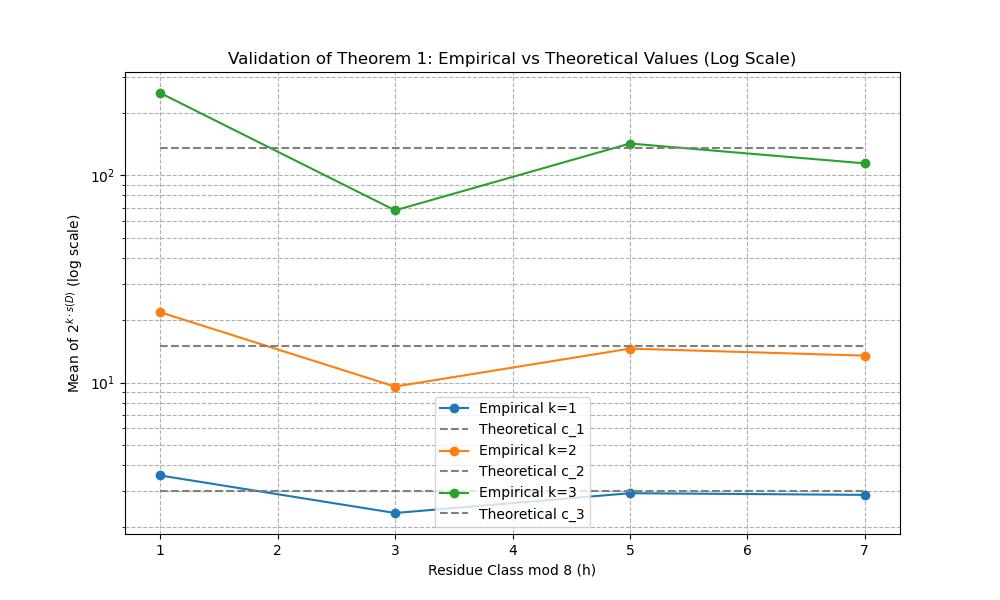}
    \caption{Moments of $2$-Selmer size on log scale}
    \label{momentslog}
\end{figure}

Again, the curves with $D$ congruent to $5, 7\pmod{8}$ are much closer to the theoretical values than the other two. 
As per the theorem, the theoretical constants $c_k$ should not depend on the residue class. 
So, it should be an interesting problem to explain the behavior of $1, 3\pmod{8}$ curves. 
It remains to be seen whether increasing the dataset size would clarify this pattern.
Is this phenomenon somehow related to the observation that there are very few congruent numbers that have residue $3\pmod{8}$ and are less than $3$ million? 
We refer the reader to our Jupyter Notebook to see the actual table of numbers. 

We now study the error term $o_k(X)$ in the statement of \Cref{hb2thm1}.
Given $k$ and $X$ define the log factor as 
\[\mathrm{lf}(k,X) := \frac{(\log\log X)^{4^k}}{(\log X)^{\frac{1}{4^k}}}.\]
Later in the paper \cite[Section 3]{hb1}, Heath-Brown provides more explicit value of the error term as $o_k(X) = O_k(X \mathrm{lf(k, X)}$.  
Our aim is to compare the difference between theoretical and empirical moments with the log factor, we would like to see if their ratio is constant as $X$ increases. 
To be precise, we first compute the following difference quantity
\[\Delta(X, k, h) := \left| \sum_{D\leq X, D\equiv h\pmod{8}}2^{k\cdot s(D)} - c_k\cdot \#S(X, h)\right|, \]
across increasing $X$. 
Then we check if $\Delta(X,k,h)/ (X \mathrm{lf}(k, X))$, which we call the normalized error, is roughly constant or not?
A plot of our empirical findings can be seen in \Cref{loglogterm}. 
It is clear that for $k=1$, the normalized error is more or less constant for $h = 1, 3$. 
However, for $k=2$ the ratio reduces with increasing $X$, indicating that there could some room for improvement. 

\begin{figure}[ht]
  \centering
  \includegraphics[width=0.65\textwidth]{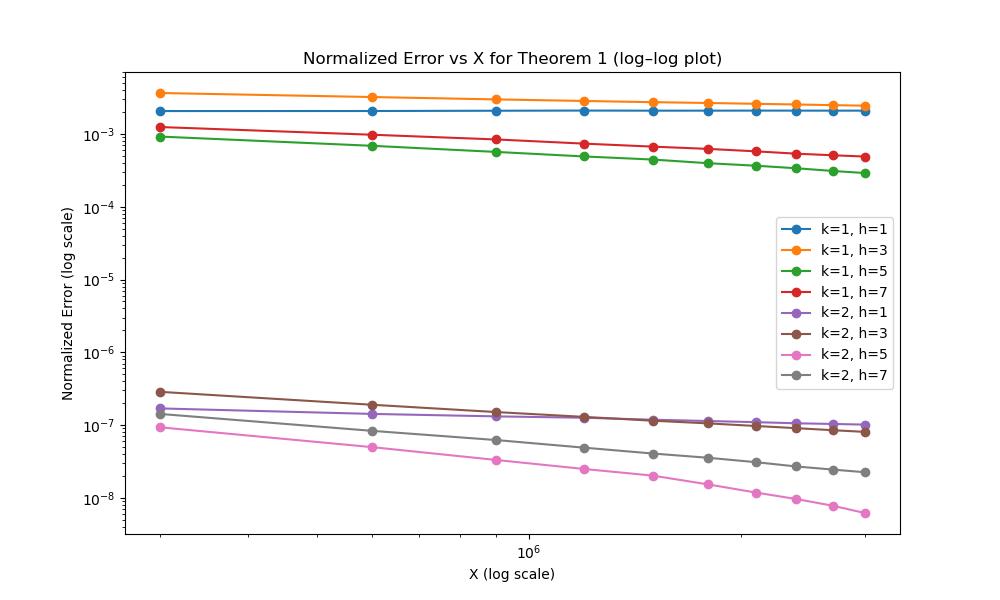}
  \caption{The error term across odd residue classes}
  \label{loglogterm}
\end{figure}

We now consider the next major result of Heath-Brown, that explicitly computes the proportion of curves, in each residue class, that has the given $2$-Selmer rank. 
In short, the result provides a formula for the probability that a given curve has rank $r$. 
The empirical probabilities are given in \Cref{selmerprobs} and a visualization of these numbers is in \Cref{prob_visual}.

\begin{thm}{\cite[Theorem 2]{hb1}}\label{s=r_thm}
Let $r$ be a non-negative integer. Then,
\[\mathrm{Prob}(\{s(D) = r\mid D\equiv h\pmod{8}\}) = \prod_{n=1}^{\infty} \frac{1}{(1+2^{-n})} \cdot\frac{2^r}{\prod_{j=1}^{r}(2^j+1)},\]
    if either $r$ is even and $h = 1$ or $3$, or $r$ is odd and $h = 5$ or $7$.
\end{thm}

\begin{table}[ht]
\centering
\begin{tabular}{|c|c|c|c|}
\hline
\textbf{Residue class} & $\mathbf{s(D)}$ & \textbf{Empirical} & \textbf{Theoretical} \\ \hline
\multirow{3}{*}{$D\equiv 1,3\pmod{8} $} & $0$ & $0.454391$ & $0.419422$ \\ \cline{2-4}
                                   & $2$ & $0.519608$ & $0.559230$   \\ \cline{2-4}
                                   & $4$ & $0.025929$ & $0.063912$ \\ \hline
\multirow{3}{*}{${D\equiv 5,7\pmod{8}} $} & $1$ & $0.856714$ & $ 0.838845$ \\ \cline{2-4}
                                   & $3$ & $0.141791 $ & $ 0.223692$ \\ \cline{2-4}
                                   & $5$ & $0.001495$ & $0.014203$ \\ \hline
\end{tabular}
\label{selmerprobs}
\caption{Probabilities per two groups of residue classes}
\end{table}

\begin{figure}[ht]
    \centering
    \begin{tabular}{cc}
        \includegraphics[clip,width=0.45\textwidth]{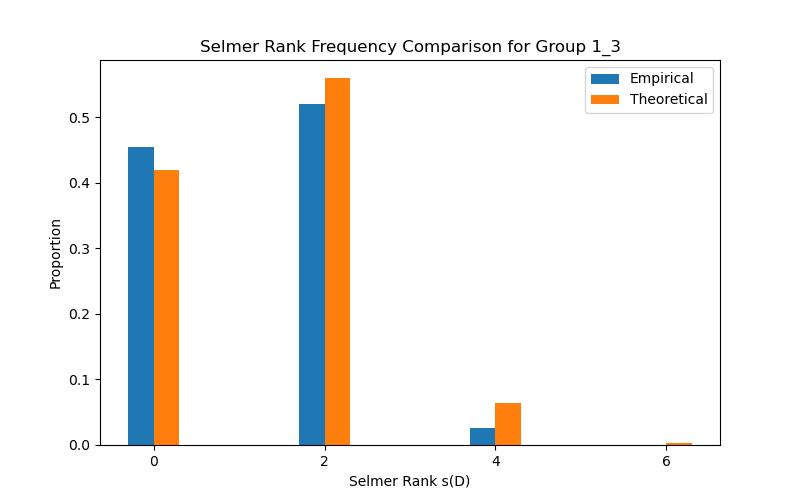} & \includegraphics[clip,width=0.45\textwidth]{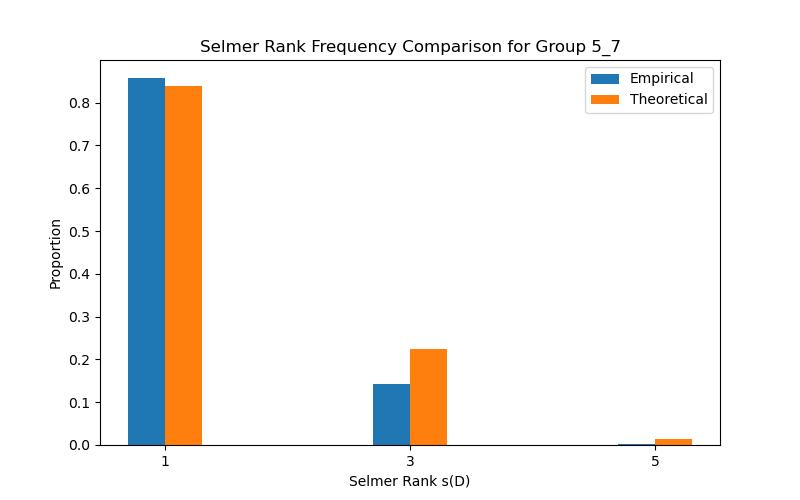}  
    \end{tabular}
    \caption{Probability distribution of $2$-Selmer ranks: empirical vs. theoretical comparison}
    \label{prob_visual}
\end{figure}

The obvious direction one can explore after this is to determine the proportion of curves with large Selmer ranks as $D$ tends to infinity. 
As the next result shows, this proportion decreases rapidly. 

\begin{cor}{\cite[Corollary 1]{hb}}\label{trailingprob}
For any fixed positive integer $r$ we have
\[\frac{\#\{D\in S(X, h)\mid s(D) \geq r\}}{\# S(X, h)} \leq 1.7313\dots \times 2^{-(r^2-r)/2} + o(1),\]
as $X\to\infty$.
\end{cor}
The empirical validation of the above result is described in \Cref{trailp}. 
Though the statement holds for every odd residue class, for brevity we have clubbed them in two groups $1, 3$ and $5, 7$. 
Note that, in each of these groups the Selmer rank has the same parity, and the cardinality $\#S(X, h)$ is almost the same for each $h$. 
However, our notebook has a cell where we compute these probabilities for every $h$ and there is not much difference in the results. 
The actual values of these trailing probabilities are in \Cref{tabletrailp}

\begin{figure}[ht]
    \centering
    \includegraphics[width=0.75\linewidth,clip]{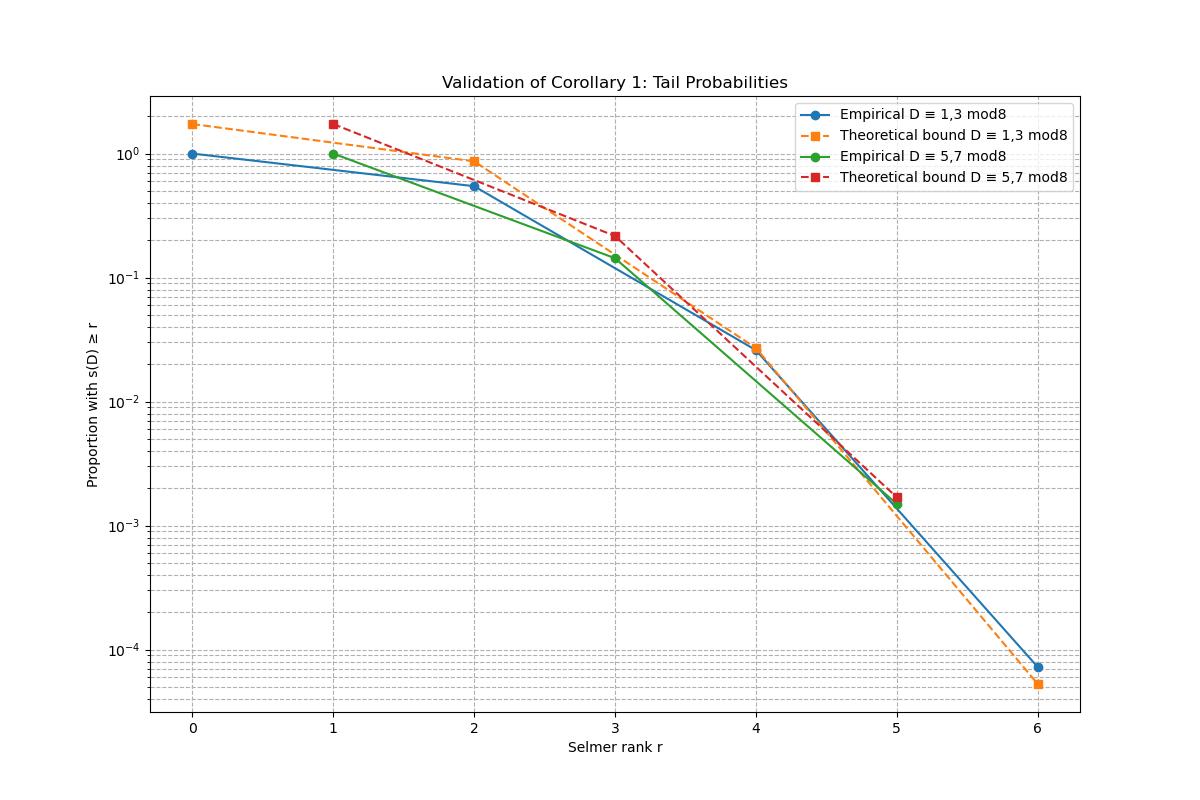}
    \caption{Trailing probabilities $2$-Selmer ranks on log scale}
    \label{trailp}
\end{figure}

\begin{table}[ht]
\centering
\begin{tabular}{|c|c|c|c|}
\hline
\textbf{Residue class} & $\mathbf{r}$ & \textbf{Empirical} & \textbf{Theoretical} \\ \hline
\multirow{3}{*}{$D\equiv 1,3\pmod{8} $} & $0$ & $1.000000$ & $1.731300$ \\ \cline{2-4}
                                   & $2$ & $0.545609$ & $0.865650$   \\ \cline{2-4}
                                   & $4$ & $0.026002$ & $0.027052$ \\ \hline
\multirow{3}{*}{${D\equiv 5,7\pmod{8}} $} & $1$ & $1.000000$ & $ 1.731300$ \\ \cline{2-4}
                                   & $3$ & $0.143286 $ & $ 0.216413$ \\ \cline{2-4}
                                   & $5$ & $0.001495$ & $0.001691$ \\ \hline
\end{tabular}
\caption{Trailing probabilities per two groups of odd residue classes}
\label{tabletrailp}
\end{table}
Additionally, the Selmer conjecture suggests that the parity of $r(D) \equiv s(D) \pmod 2$. 
Thus, under the assumption of BSD, one should have that the parity of $r_{an}(D)$ is the same as that of $s(D)$. 
We have the following corollaries about the average of the $2$-Selmer ranks $s(D)$.
\begin{cor}{\cite[Corollary 2]{hb}}
    For any odd integer $h$, we have 
    \[\sum_{D \in S(X,h)} {s(D)} \le \frac{3}{2}\#S(X,h) + O(X(\log X)^{-1/4}(\log \log X)^{8}).\]
    Further, assuming that $s(D)$ and $r_{an}(D)$ always have the same parity, we have
    \[\sum_{D \in S(X,h)} {s(D)} \le \frac{4}{3}\#S(X,h) + O(X(\log X)^{-1/4}(\log \log X)^{8}).\]
    Hence, $s(D)=0$ for at least $1/3$-rd of all $D \equiv 1, 3 \pmod 8$ and $s(D)=1$ for at least $5/6$-th of all $D \equiv 5, 7 \pmod 8$.
\end{cor}
The above statement was improved by Heath-Brown in the subsequent paper, as we now state. 
First, define constants $c_h'$as follows:
\begin{align*}
    c'_h \approx \begin{cases}
        1.2039\dots \text{ if } h = 1, 3\\ 
        1.3250\dots \text{ if } h = 5, 7
    \end{cases}
\end{align*}

These constants turn out to be average Selmer rank values, as stated in the Corollary below.

\begin{cor}{\cite[Corollary 2]{hb1}}\label{c'c''cor}
    We have
    \[{\sum_{D \in S(X,h)} s(D)} = c'_h{\#S(X,h)} + o(X), \text{ as } X \to \infty.\]
\end{cor}

The data reveals the following averages for the $1.2$ million curves corresponding to odd residue classes.

\begin{table}[ht]
    \centering
    \begin{tabular}{|c|c|c|c|}
    \hline
    \textbf{Residue} & \textbf{Empirical} & \textbf{Theoretical} & \textbf{Count}\\ \hline
        $1$ & $1.4511$ & $1.2039$ & $303961$ \\ \hline
        $3$ & $0.8356$ & $1.2309$ & $303961$ \\ \hline
        $5$ & $1.2961$ & $1.3250$ & $303959$ \\ \hline
        $7$ & $1.2830$ & $1.3250$ & $303963$ \\ \hline
    \end{tabular}
    \caption{Average $2$-Selmer rank across odd residue classes}
    \label{avgselmer1}
\end{table}
Empirical averages are much closer in the residue classes $5$ and $7$, whereas in the residue class $3$ it is much less and higher than the theoretical prediction in the class of $1$. 
If, instead, we combine the residue classes in two groups then we get averages that are much closer to the predicted values. 

\begin{table}[ht]
    \centering
    \begin{tabular}{|c|c|c|}
    \hline
    \textbf{Residue} & \textbf{Empirical} & \textbf{Theoretical} \\ \hline
        $1, 3$ & $1.143367$ & $1.2039$  \\ \hline
        $5, 7$ & $1.289563$ & $1.3250$  \\ \hline
    \end{tabular}
    \caption{Average $2$-Selmer rank across two groups of odd residue classes}
    \label{avgselmer2}
\end{table}
In our data, the proportion of $D\equiv 1, 3\pmod{8}$ such that $s(D) = 0$ has come out $0.454390$, where the theoretical prediction is $\frac{1}{3}$. 
Whereas the proportion of $D\equiv 5, 7\pmod{8}$ such that $s(D) = 1$ is equal to $0.856713$, which is much closer to the theoretical prediction of $\frac{5}{6}$. 

\subsection{The Poonen-Rains heuristics}
In this subsection, we compare the results and conjectures of Poonen and Rains \cite{pr} with our empirical data.
Conjecture 1.1 in their paper predicts that, as we vary over all elliptic curves $E$ over $\mathbb{Q}$ ordered by height,

\[\mathrm{Prob}(\dim_{\F_2}\Sel_2(E/\Q) = d) = \left(\prod_{j\geq 0} (1+2^{-j})\right) \left(\prod_{j=1}^d \frac{2}{2^j - 1}\right). \]

The expression is similar to that of \Cref{s=r_thm}, where, the probabilities were calculated for the congruent number curves partitioned into residue classes of $\pmod{8}$. 
However, we calculate these probabilities for all the $1.8$ million congruent number curves in our database. 
The striking observation is that even in this restricted family the probability mass functions are fairly close to the theoretical prediction meant to be valid for all possible elliptic curves; see  \Cref{prpmf}

\begin{figure}[ht]
    \centering
    \includegraphics[width=0.9\linewidth]{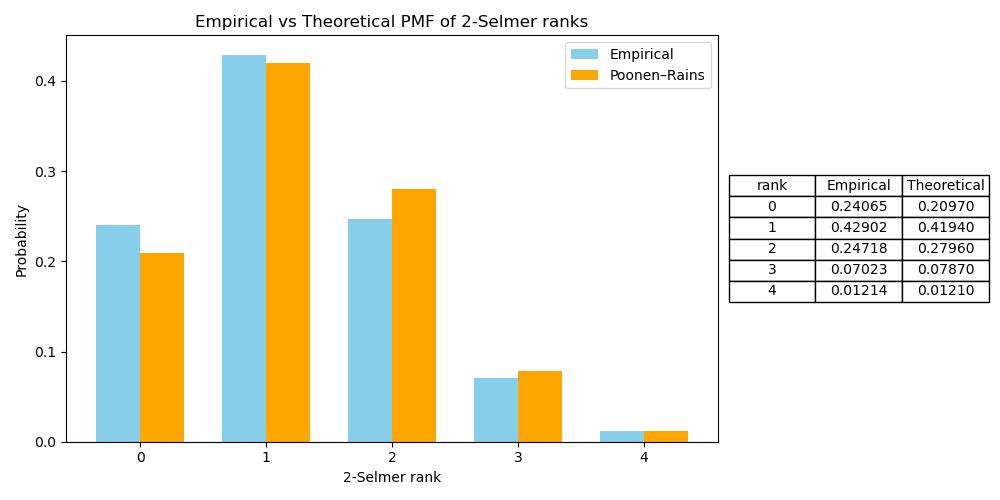}
    \caption{$2$-Selmer probabilities for all congruent number elliptic curves}
    \label{prpmf}
\end{figure}


\begin{prop}{\cite[Proposition 2.22]{pr}}
    Let $m \ge 0$ be a fixed integer. Then we have that the average of the $m$-th moments of the size of $2$-Selmer groups of $E/\Q$, i.e., the average of  $\left(\# \Sel_2(E/\Q)\right)^m$ is given by $\prod_{i=1}^{m} (2^i+1)$ as $E$ varies over all elliptic curves over $\Q$.    
    In particular, the average size of $\Sel_2(E/\Q)$ over all $E/\Q$ is $3$.     
    One also has that the probability that $\dim_{\F_2} \Sel_2(E/\Q) \equiv 0 \pmod 2$ is $1/2$.
\end{prop}

\Cref{prmoments} shows that when one studies the family of congruent number curves, the moments of their $2$-Selmer sizes are fairly close to those of all elliptic curves over $\Q$. 

\begin{figure}[ht]
    \centering
    \includegraphics[width=0.75\linewidth]{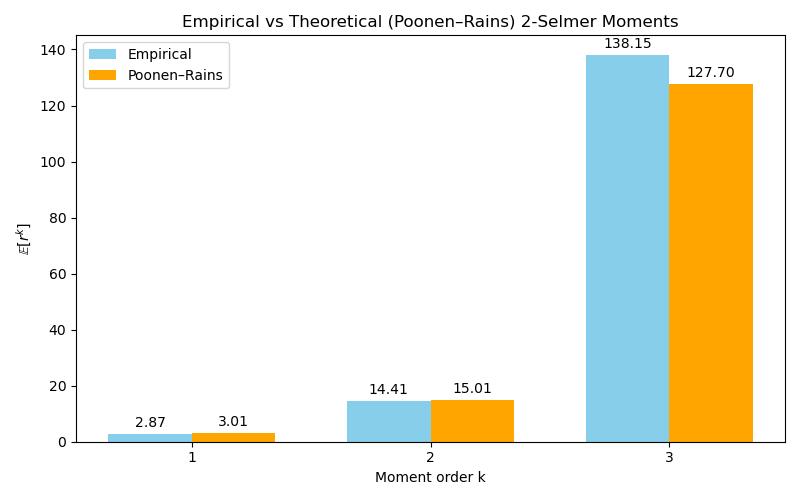}
    \caption{Moments of 2-Selmer size for all congruent number elliptic curves}
    \label{prmoments}
\end{figure}

\subsection{Delaunay Heuristics}\label{delaunay}
We recall the heuristics by Delaunay \cite{del} on the Tate-Shafarevich groups of elliptic curves. 
Let $\Sh(E/\Q)$ denote the Tate-Shafarevich group of $E/\Q$ (just $\Sh$, if the context is clear). 
Delaunay’s heuristics propose a probabilistic model for the distribution of the $p$-primary part of $\Sh[p^\infty]$.
Assuming that $\Sh$ is a finite abelian group with a canonical, nondegenerate, alternating, bilinear pairing, Delaunay adapts the measure on isomorphism classes: instead of counting automorphisms of the group, only symplectic automorphisms (those preserving the pairing) are considered.
In particular, the work conjecturally predicts, for example, the probability that $\Sh$ is cyclic, or isomorphic to a particular group, or that a prime divides its order.
However, for our experiments we focus only on computing probability that the $\F_p$-dimension of the $p$-torsion, $\Sh[p]$, is the given even integer.  
To that effect, we have the following formula \cite[Example F]{del}, see also \cite[Section 3.2]{bhkssw} and \cite[Conjecture 5.1]{pr}. 
\[\mathrm{Prob}(\dim_{\F_p}\Sh[p] = 2n) = p^{-n(2r+2n-1)} \frac{\prod_{i= n+1}^{\infty} (1-p^{-(2r+2i-1)})}{\prod_{i=1}^n(1-p^{-2i})},\]
where $r\in\{0,1\}$ is the MW rank.

\begin{figure}[ht]
    \centering
    \includegraphics[width=0.7\linewidth]{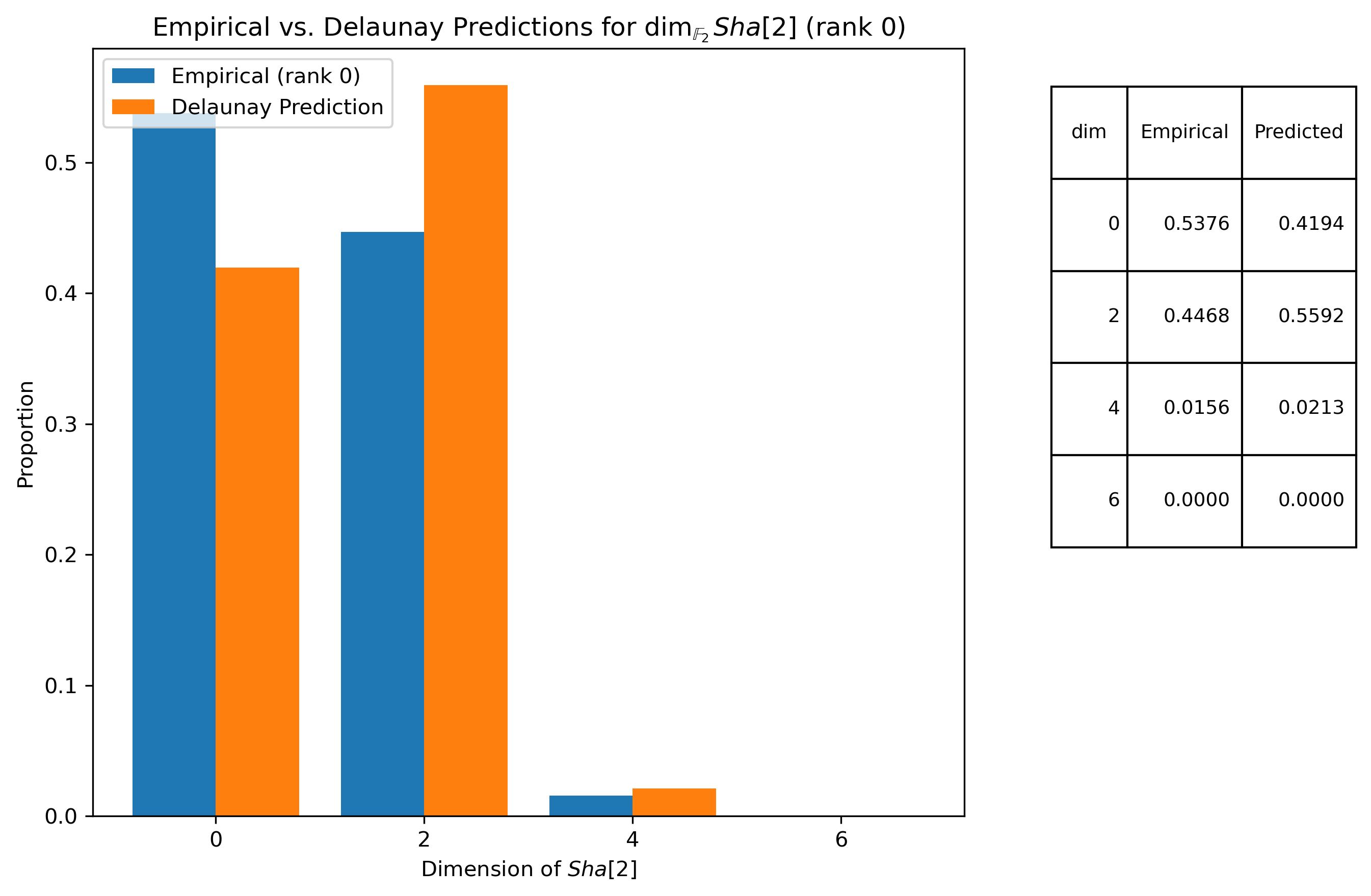}
    \caption{$\Sh[2]$ probabilities for rank $0$ curves}
    \label{sha2r0}
\end{figure}

\begin{figure}[ht]
    \centering
    \includegraphics[width=0.7\linewidth]{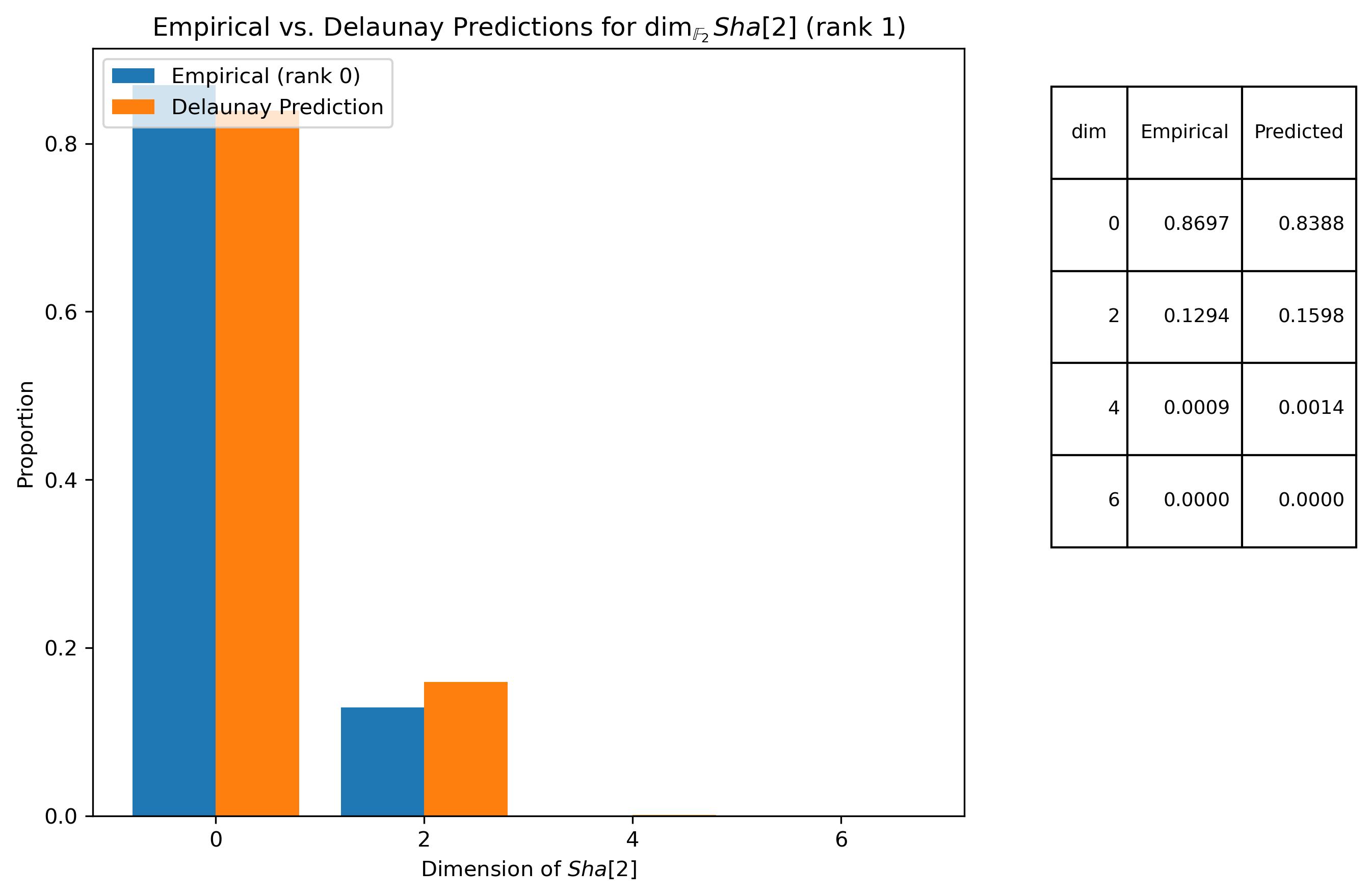}
    \caption{$\Sh[2]$ probabilities for rank $1$ curves}
    \label{sha2r1}
\end{figure}

Recall the fundamental exact sequence \eqref{eq:selshaseq} from which one can calculate the $p$-rank of $p$-torsion of $\Sh$ for congruent number curves: 
\[\dim_{\F_p} \Sh(E_D/\Q)[p] = \dim_{\F_p} \mathrm{Sel}_p(E_D/\Q) - \text{rank } E_D(\Q) - \dim_{\mathbb{F}_p}E(\Q)[p]. \]

For our experiments we consider $p = 2, 3$. 
Empirically, we observe that the Selmer rank and the MW rank have the same parity for congruent number curves. 
Also, recall that the dimension of $E_D(\Q)[2]$ is $2$ and that of $E_D(\Q)[3]$ is zero, since there are no points of order $3$. 
Hence, the rank of $p$-torsion of $\Sh(E_D/\Q)$ is always even. 
We have MW rank calculated for curves corresponding to square-free numbers up to $1$ million. 
We also have $2$-Selmer ranks calculated for these curves. 
However, we have $3$-Selmer ranks of only $59,110$ curves. 
Then for each $p$, we isolate curves into two distinct groups: rank $0$ and rank $1$. 
Then for each group, we compute the corresponding proportions and then compare the data. 
\Cref{sha2r0} shows the comparison of empirical and theoretical probabilities of $\Sh(E_D/\Q)[2]$ ranks in the case of rank $0$ curves. 
Visually it is clear that in the rank $0$ case the distribution of these ranks is not close to theoretical predictions. 
In the rank $1$ case the distributions almost match, as seen in \Cref{sha2r1}. 
The story is not very different in case of $\Sh(E_D/\Q)[3]$ ranks as seen in \Cref{sha3r0} and \Cref{sha3r1}. 
\begin{figure}[ht]
    \centering
    \includegraphics[width=0.7\linewidth]{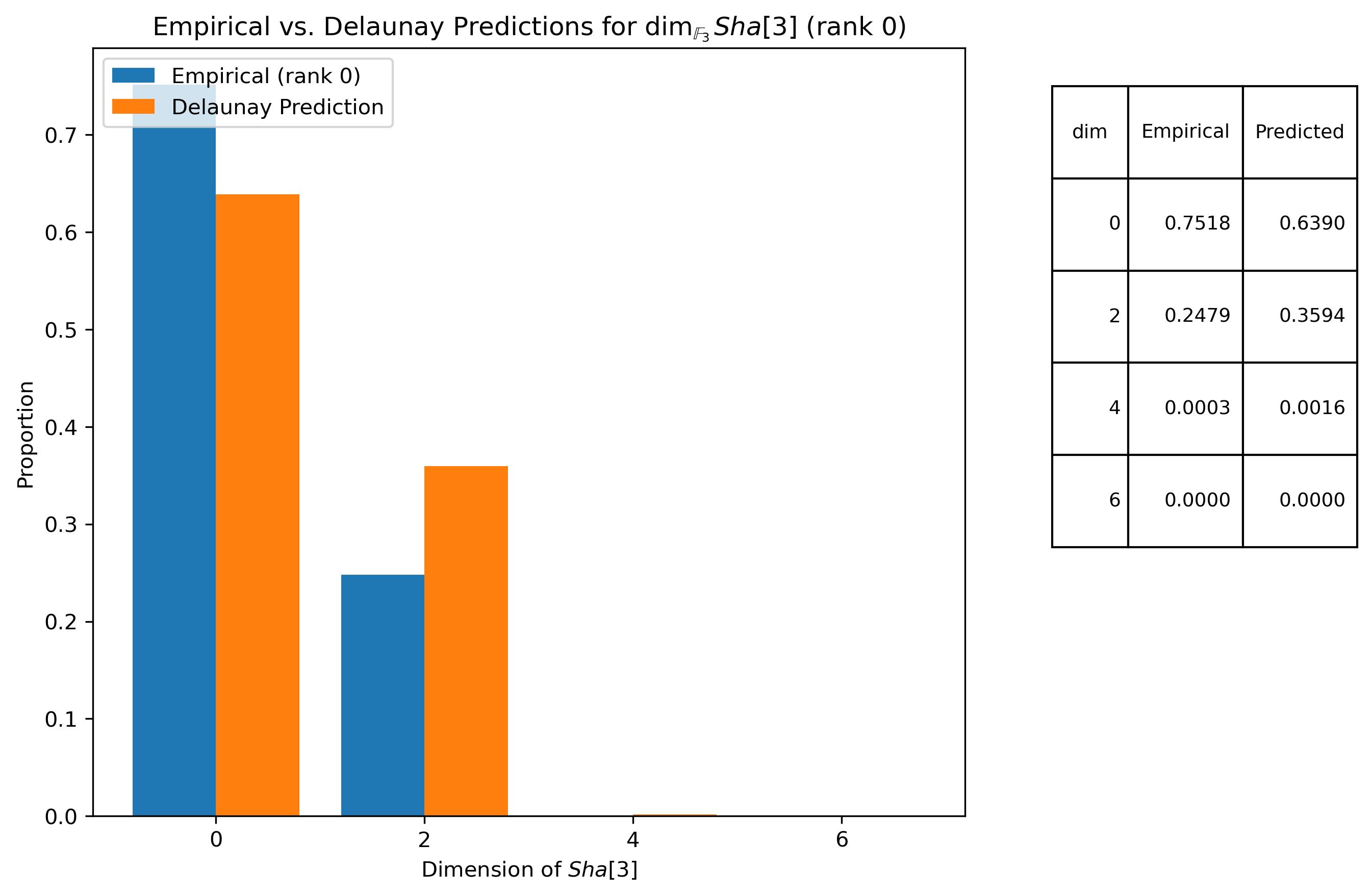}
    \caption{$\Sh[3]$ probabilities for rank $0$ curves}
    \label{sha3r0}
\end{figure}

\begin{figure}[ht]
    \centering
    \includegraphics[width=0.7\linewidth]{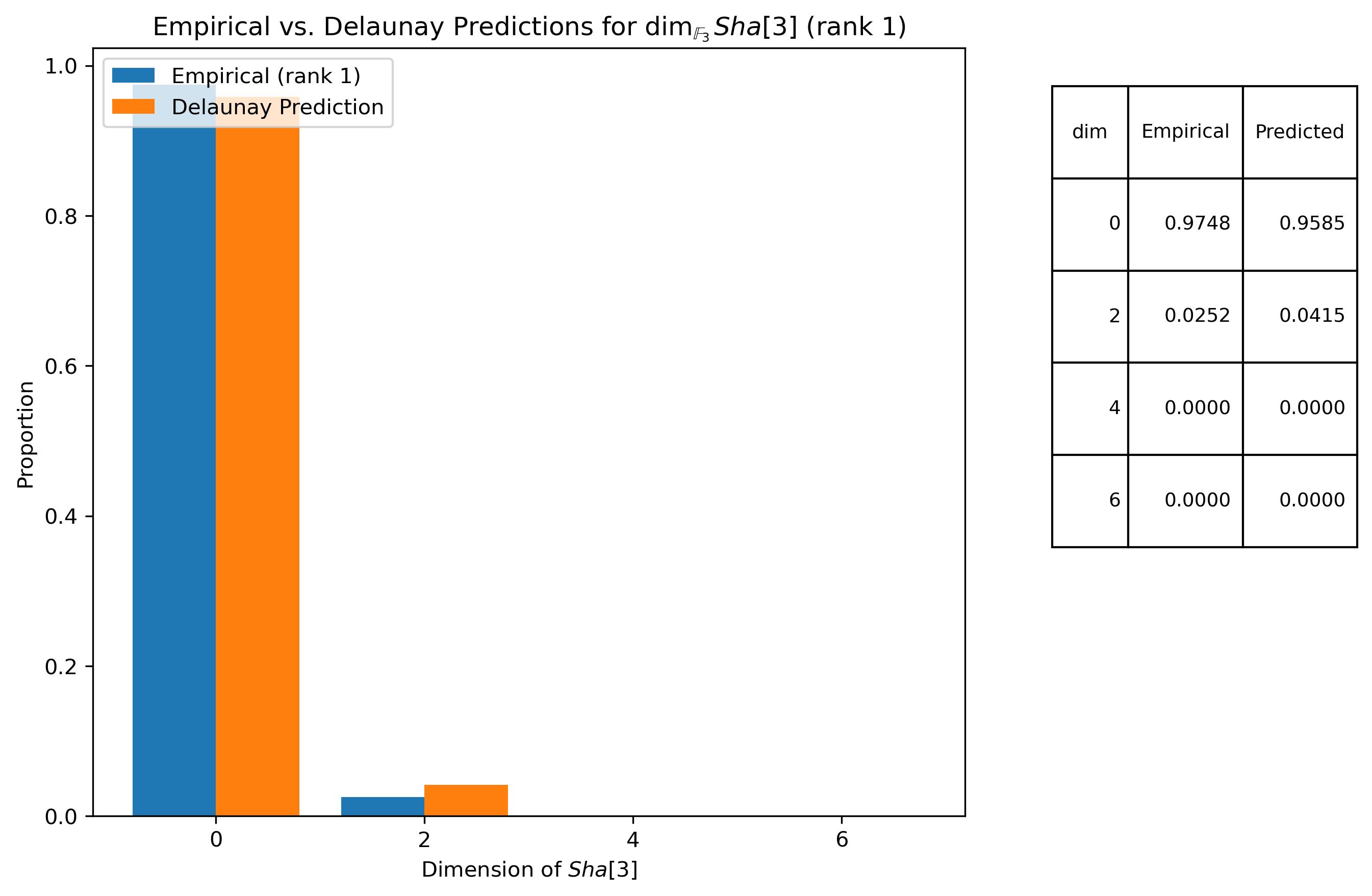}
    \caption{$\Sh[3]$ probabilities for rank $1$ curves}
    \label{sha3r1}
\end{figure}
We note that the other heuristics conjectured by Delaunay, about the Tate-Shafarevich group itself, are difficult to verify, since one cannot determine that group exactly. 
However, interested reader can devise experiments using analytically computed orders, assuming BSD.
A similar treatment can be found in \cite[Section 4.3]{shaorder}.
Our dataset currently has analytic order of $\Sh$ computed for curves corresponding to numbers up to $1$ million.

\section{The Goldfeld conjecture and the BSD conjecture}\label{gold-smith}
\subsection{Goldfeld conjecture}
For an elliptic curve $E:y^2=x^3+ax+b$ over $\Q$ and $d \in \Q^*/\Q^{*2}$, denote by $E^{(d)}:y^2=x^3+ad^2x +bd^3$ the quadratic twist of $E$ by $d$. Recall that the $L$-function associated to an elliptic curve $E$ with conductor $N_E$ is given by 
\[L_E(s) = \prod_{p \nmid N_E} (1-a_p p^{-s} + p^{1-2s})^{-1} \prod_{p \mid N_E} (1-a_p p^{-s})^{-1}\]
and that the order of vanishing of $L_E(s)$ at $s=1$ is called the \emph{analytic rank} of $E$, denoted $r_{an}(E)$. 
The Goldfeld conjecture \cite{gold} asserts that in the family $\mathcal{E} := \left\{E^{(d)}\right\}$, the analytic rank $r_{an}(E^{(d)})=0$ (resp. $r_{an}(E^{(d)})=1$) for $50\%$ of $d$'s. The occurrences of $r_{an}(E^{(d)}) \ge 2$, although infinitely often, should only account for $0\%$ of $d$'s. More precisely,
\begin{conj}[Goldfeld]
    For $X>0$ and a fixed elliptic curve $E/\Q$, let
    \[N_r(X):= \left|\left\{|d|<X \mid  r_{an}(E^{(d)}) = r\right\}\right|.\]
    Then for $r \in \{0,1\}$,
    \[N_r(X) \sim \frac{1}{2} \sum_{|d|<X} 1, \text{ as } X \to \infty.\]
\end{conj}
 Smith \cite{smithgoldfeld} announced a proof of the Goldfeld's conjecture, conditional to BSD, for curves with full $2$-torsion. Unconditionally, for the families of CM elliptic curves, we know that the conjecture holds for $r=0$ when the CM field is not $\Q(\sqrt{-2})$ and for $r=1$ when $2$ splits in the CM field. Thus, the Goldfeld conjecture holds for the family of curves containing a curve of conductor $49$. However, for all other quadratic twist families of elliptic curves, the full strength of this conjecture is not known yet. Therefore, one often considers a weaker version, by replacing $50\%$ by any positive proportion.
 \begin{conj}[Weak Goldfeld]
     For $r \in \{0,1\}$, \[\liminf_{X\to\infty}\frac{N_r(X)}{|\{ |d|<X\}|}  >0.\]
 \end{conj}
Heath-Brown \cite{hb2} proved that the weak Goldfeld holds, conditional to GRH. The recent works of Smith \cite{smith2016congruent} and Tian-Yuan-Zhang \cite{tyz} showed that it is true for the congruent number curves family.

We now see what the data shows regarding distribution of analytic ranks. 
These ranks are calculated using both Sage and Magma functions. 
We use the \texttt{rank()} function of Sage twice with two different algorithm input parameters. 
First we set the algorithm to \emph{all}, that forces the function to use all the three, the PARI library function, Watkins' \emph{sympow} method and Rubenstein's \emph{lcalc} method to calculate the ranks and return the common answer. 
Second, we set the algorithm to \emph{magma}, that calls the function in MAGMA. 
We compute an upper bound for the analytic rank using Sage's \texttt{analytic\_rank\_upper\_bound()} function. 
We also compute the root number of each elliptic curve. 
In case the outputs from all the (four) algorithms match then declare it as the analytic rank. 
Otherwise, we use either the root number or the calculated upper bound or even the MW rank, if needed. 
However, the calculations done so far imply that all rank algorithms give the same value. 

Out of $98,400$ curves, there are $43,529$ curves of rank $0$, $48,917$ curves of rank $1$, $5,665$ curves have rank $2$, $285$ curves of rank $3$ and exactly $4$ are of rank $4$. 
A bar chart of frequency distribution is shown in \Cref{anarank}

\begin{figure}[ht]
    \centering
    \includegraphics[width=0.5\linewidth]{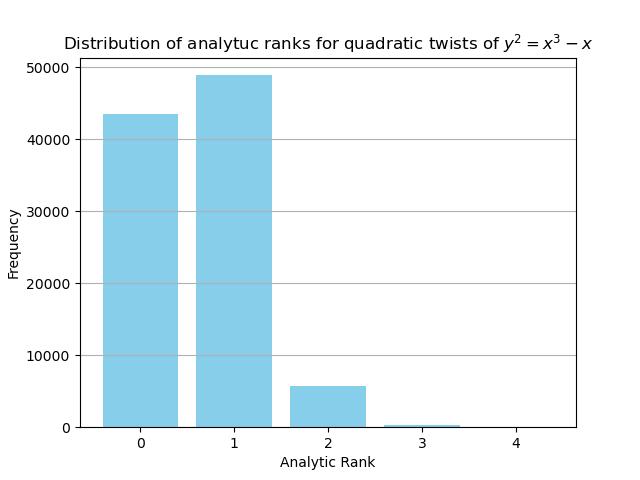}
    \caption{Frequency distribution of analytic ranks}
    \label{anarank}
\end{figure}

It is clear that in the available empirical data the proportion is not $50/50$ as predicted by Goldfeld. 
\Cref{running} shows how for each rank $0$ and $1$ the running proportions. 
For rank $1$ curves the proportion quickly approaches $0.49$ for rank $0$ curves it approaches $0.46$. 
We employ two statistical tests, the $\chi^2$-test and the Fischer's test to see how the observed split between analytic ranks $0, 1$ deviate from the $50/50$ split predicted by the conjecture. 

\begin{figure}[ht]
    \centering
    \includegraphics[width=0.5\linewidth]{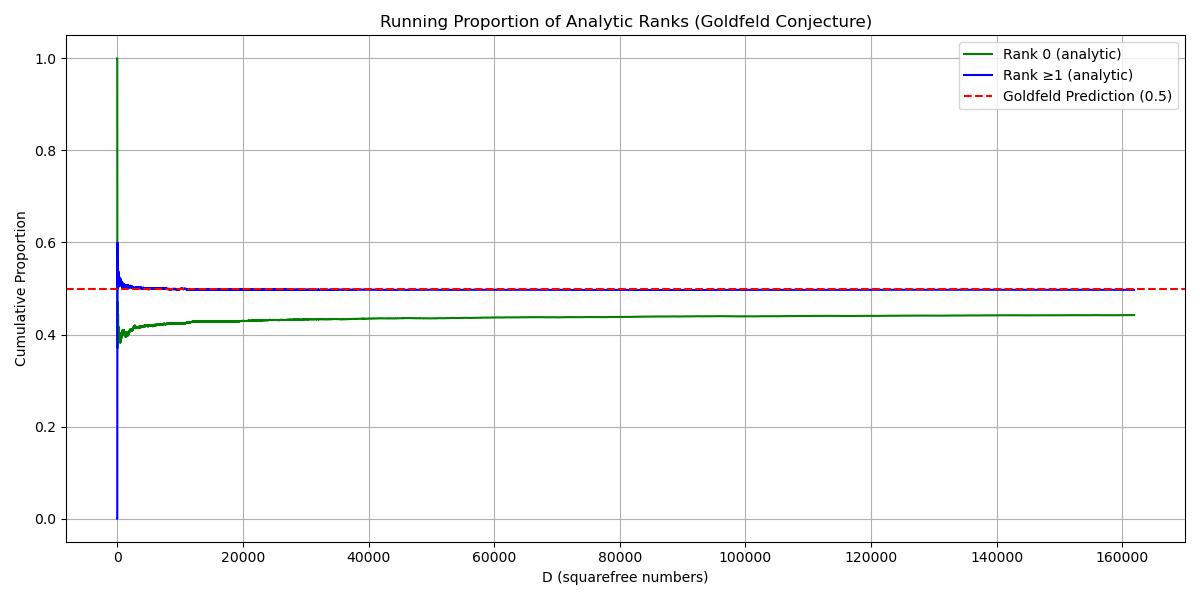}
    \caption{Proportions of analytic ranks with increasing value of $D$}
    \label{running}
\end{figure}

The following table shows the actual empirical split with the growing sample size and the $p$-values corresponding to each of the two tests. 
Recall that, the $p$-value is the probability that we would see a result as extreme as the one we observed, assuming that the Goldfeld's conjecture is true. 

\begin{table}[ht]
    \centering
    \begin{tabular}{|c|c|c|c|c|}\hline
    $D$ maximum & sample size & rank $0$ & rank $1$ & $p$-value \\ \hline
        $16446$ & $10000$ & $4280$ & $5720$ & $0$ \\ \hline
        $32905$ & $20000$ & $8670$ & $11330$ & $0$ \\ \hline
        $49337$ & $30000$ & $13058$ & $16942$ & $0$ \\ \hline
        $65797$ & $40000$ & $17510$ & $22490$ & $0$ \\ \hline
        $82257$ & $50000$ & $21932$ & $28068$ & $0$ \\ \hline
    \end{tabular}
    \caption{The table of proportions and $p$-values}
    \label{goldcomparison}
\end{table}

Both the statistical tests imply that assuming Goldfeld's conjecture is correct, such an uneven split in ranks in the empirical data would almost never happen — unless something else is going on.
Hence, we are still in the finite-sample regime, and the convergence to $50/50$ hasn’t kicked in.
Hence, we ran the experiments using the MW ranks (since BSD would imply that both the ranks are equal) which were calculated for more than $600000$ curves. 
However, there is no difference, the deviation is still large and the distributions are far from the predicted $50/50$ split.

\subsection{Density of congruent numbers and the BSD conjecture}
Most results about proportions of congruent numbers and Selmer ranks were mostly in the form of heuristics. 
However, recent seminal work of Smith \cite{smith2016congruent} converts some of the heuristics and conjectures into provable theorems. 
Intuitively, it is shown that congruent numbers aren't rare or exceptional — in fact, they have positive density among square-free integers.
It is also proved that for a large subfamily of congruent number curves, one can verify the full BSD conjecture, even though it remains unproven in general.
It is one of the first unconditional density results of this kind. 
The main technique involves translating the problems to probabilistic behavior of matrices over $\mathbb{F}_2$. 
Using our empirical data we verify two of his results, the first is about the BSD conjecture and the other is about the density of congruent numbers. 

We begin by introducing the BSD formula in the context of congruent number curves. 

\begin{defn}
    The normalized BSD quantity of a congruent number (CN) elliptic curve $E_D$ is defined as 
    \[\mathcal{L}(E_D) := \frac{16\cdot  L(1, E_D)}{\Omega(E_D)\cdot \prod_p c_p}, \]
    where $L(1, E_D)$ is the special value, $16$ is the square of the torsion order (which is uniformly $4$ for all the CN curves), $\Omega(E_D)$ is the real period and $c_p$ are the local Tamagawa factors. 
\end{defn}
Note that $\mathcal{L}(E_D)$ is nothing but the analytic order of Sha. 
A weak form of the full BSD conjecture can be stated for congruent number curves as follows; this is the main theorem of Smith. 
\begin{thm}{\cite[Theorem 1.2]{smith2016congruent}}\label{bsdthm}
    Let $E_D$ be a congruent number curve corresponding to a square-free integer $D$. 
    The normalized BSD quantity $\mathcal{L}(E_D)$ is an odd integer if and only if the $2$-Selmer group is generated by the rational $2$-torsion subgroup of $E_D$. 
\end{thm}

In order to computationally verify above result we compute the BSD parameters, Tamagawa product, regulator, special value and the real period for the congruent number curves using Sage Math. 
Our current database consists of $195,800$ curves and their $4$ BSD parameters. 
For each curve we compute its normalized BSD quantity and round it off to the nearest integer. 
We add a boolean column called \emph{L\_BSD\_odd} to the data frame, whose value is set to True if the corresponding BSD quantity is odd and False otherwise. 
The result is verified if the corresponding $2$-Selmer rank is $2$. 
Empirically, the result holds for almost $25\%$ of the curves. 
\Cref{bsdv1} shows the proportion of curves per residue classes for which \Cref{bsdthm} holds. 
It holds for $31.6\%$ curves in $1\pmod{8}$, for $54.74\%$ curves in $2\pmod{8}$ and for $61.65\%$ of the curves in $3\pmod{8}$ residue class. 
As expected, the result doesn't apply to the curves in the remaining residue classes. 

\begin{figure}[ht]
    \centering
    \includegraphics[width=0.7\linewidth]{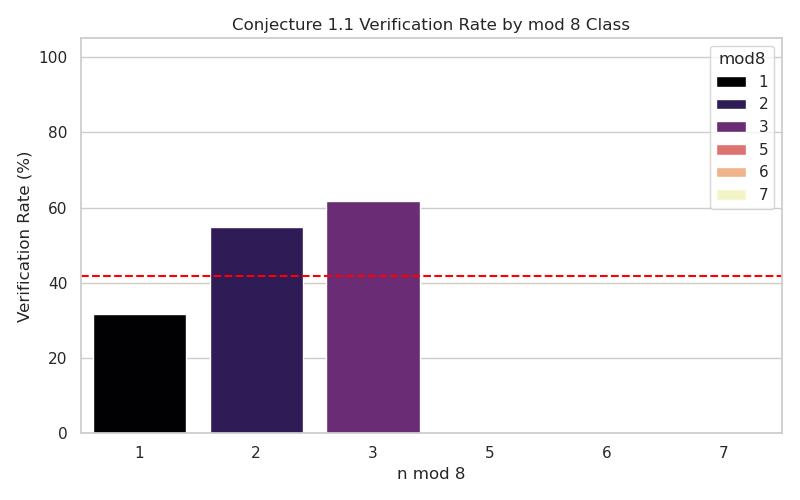}
    \caption{BSD verification across residue classes}
    \label{bsdv1}
\end{figure}

Recall that the torsion subgroup of $E_D(\Q)$ is $\mathbb{Z}/2\times \mathbb{Z}/2$. 
Hence, the $2$-Selmer rank is $2$ and equivalently the MW rank is $0$. 
This implies that the theorem doesn't apply to those $D$ which are congruent to $5,6,7\pmod{8}$. 
Further, combining Waldspurger's formula with random matrix models, Smith concludes that the condition that the Selmer group is generated by $2$-torsion also implies that the full BSD holds. 
In particular, he proves the following

\begin{cor}{\cite[Corollary 1.3]{smith2016congruent}}\label{bsdcor}
    The congruent number curve $E_D$ satisfies the full BSD conjecture if the corresponding Selmer group is generated by the torsion subgroup. 
    In particular, the curve $E_D$ satisfies full BSD for at least $41.9\%$ of square-free positive $D\equiv 1, 2, 3\pmod{8}$. 
\end{cor}
In order to verify the above result, one needs to calculate the how many curves in the residue classes $1, 2, 3$ have Selmer rank exactly equal to $2$?
In our database there are $97893$ curves with even Selmer rank, out of those $48305$ curves are of rank $2$. 
This is slightly more than $49\%$. 
\Cref{bsdv1} shows the proportion across the three residue classes. 
In the main databse as well we have similar figures, see for example \Cref{sel_distro}.

Smith also proves that a positive proportion (not all) of positive integers $\equiv 5, 6, 7\pmod{8}$ are congruent numbers \cite[Theorem 1.5]{smith2016congruent}. 
In particular, at least $62.9\%$ of square-free $D\equiv 5, 7\pmod{8}$ are congruent numbers and at least $41.9\%$ of square-free $D\equiv 6\pmod{8}$ are congruent numbers. 
However, our database shows that $100\%$ of these numbers are congruent, consistent with the conjecture, which states that all such numbers are indeed congruent. 

Finally, as for the case, when one considers only those $D$ which are primes and either $5$ or $7 \pmod 8$, these are known to be congruent. This result is verified using our data of $39292$ curves $E_D$ with $D$ a prime $\equiv 5,7 \pmod 8$. The rank of $E_D(\Q)$ was observed to be positive in all the cases. On the other hand, when $D$ is a prime congruent to $3 \pmod 8$, then $D$ is non-congruent is well-known. We verify this result using our dataset of $19653$ values of such $D$. For all of these curves where $D \equiv 3 \pmod 8$ is a prime, the rank of $E_D(\Q)$ is $0$.

\section{The average value of Frobenius traces}\label{sec:murmurations}
In a recent work \cite{mumur01} He, Lee, Oliver and Pozdnyakov  discovered an unexpected and visually striking pattern that appears when one averages the Frobenius traces over a large collection of elliptic curves.
They called this pattern “murmurations of elliptic curves” due to the resemblance of their graphs to the swarming patterns of flocks of starlings.
To the best of our knowledge, the average value of Frobenius traces of congruent number curves (or in general in any quadratic twist family) has not been studied. 
Hence, we devised an experiment to compute Frobenius traces of these curves and calculate the average for each prime. 
In this section, we describe our findings. 

For a prime $p \nmid 2D$, denote by $\tilde{E}_D$ the reduced curve over the finite field $\F_p$. 
Then $\#\tilde{E}_D(\F_p) = p+1-a_p$, where $a_p=a_p(E_D)$ is the trace of Frobenius of $E_D$ modulo $p$. 
It is known, in general, that $|a_p| \le 2\sqrt{p}$, and in particular, for the curves $E_D$, we have the following result:
\begin{prop}\label{apformula}
Let $E_1$ be the curve $y^2=x^3-x$ and $p$ be an odd prime. Then 
    \[a_p(E_D)=\begin{cases} 0, & \text{if } p \equiv 3 \pmod 4,\\ \left( \frac{D}{p} \right) a_p(E_1), & \text{if } p \equiv 1 \pmod 4,   
\end{cases}\]
where $\left( \frac{D}{p} \right)$ is the Legendre symbol for $D$ modulo $p$.

Moreover, we have that $a_p(E_D)$ is even for all $p>2$.
\end{prop} 
\begin{proof}
    When $p \equiv 3 \pmod 4$, then $\tilde{E}_D$ has super-singular reduction over $\F_p$ and so we have $\#\tilde{E}_D(\F_p) = p+1$ (see \cite[Ch. X.6]{sil}).

    On the other hand, when $p \equiv 1 \pmod 4$ and $\left( \frac{D}{p} \right) =1$, then $\tilde{E}_1 \cong \tilde{E}_D$ over $\F_p$ and clearly, $\#\tilde{E}_1(\F_p) = \#\tilde{E}_D(\F_p)$. Next, when $p \equiv 1 \pmod 4$ and $\left( \frac{D}{p} \right) = -1$, then for any given $x \in \F_p$, there exists a $y \in \F_p$ such that $(x,y)$ lies on either $\tilde{E}_1$ or $\tilde{E}_D$ (or both). Also, for such a $y$, $-y$ is also a solution. This implies that $\#\tilde{E}_1(\F_p) + \#\tilde{E}_D(\F_p) = 2p+2$, which in turn implies that $a_p(E_1)=-a_p(E_D)$.

Furthermore, since $E_1$ has a $\Q$-rational point of order $2$, we get that $a_p(E_1) \equiv p+1 \pmod 2$ for any prime $p>2$, i.e. $a_p(E_1)$ is even for all $p>2$. Consequently, $a_p(E_D)$ will be even for all $p>2$.
\end{proof}

Given a large positive integer $X$, let $\mathrm{SF}(X)$ to be the set of all positive square-free integers less than or equal to $X$. 
Enumerate the primes in ascending order as $p_1=2, p_2=3, p_3=5, \ldots$ and so on.
The Frobenius average function $f_{X}:\mathbb{N}\to\mathbb{R}$ is defined as follows:
\begin{align*}
    f_{X}(n) &:= \frac{1}{\#\mathrm{SF}(X)}\sum_{D\in \mathrm{SF}(X)} a_{p_n}(E_D)\\
    & = \frac{a_{p_n}(E_1)}{\#\mathrm{SF}(X)}\sum_{D\in \mathrm{SF}(X)} \left( \frac{D}{p_{n}} \right).
\end{align*}

It is clear that if $p_n\equiv 3\pmod{4}$ then $f_{X}(n) = 0$, for all $X,$  and for other primes, it is the average of Legendre symbols times the Frobenius trace of $E_1$ at $p_n$. 

Following the strategy of \cite{mumur01}, first we computed averages only for the rank $0$ curves, but, we did not see any nice patterns. 
Next, we computed averages for curves with nonzero rank (i.e., corresponding $D$ is a congruent number). 
The plot can be seen in \Cref{traceavgrank}, the black dots correspond to rank $0$ curves and the red dots correspond to nonzero rank curves. 
There are two interesting patterns: first, there is a reflection symmetry across the $X$-axis, second, the average values are very close to zero. 

\begin{figure}[ht]
    \centering
    \includegraphics[width=0.7\linewidth]{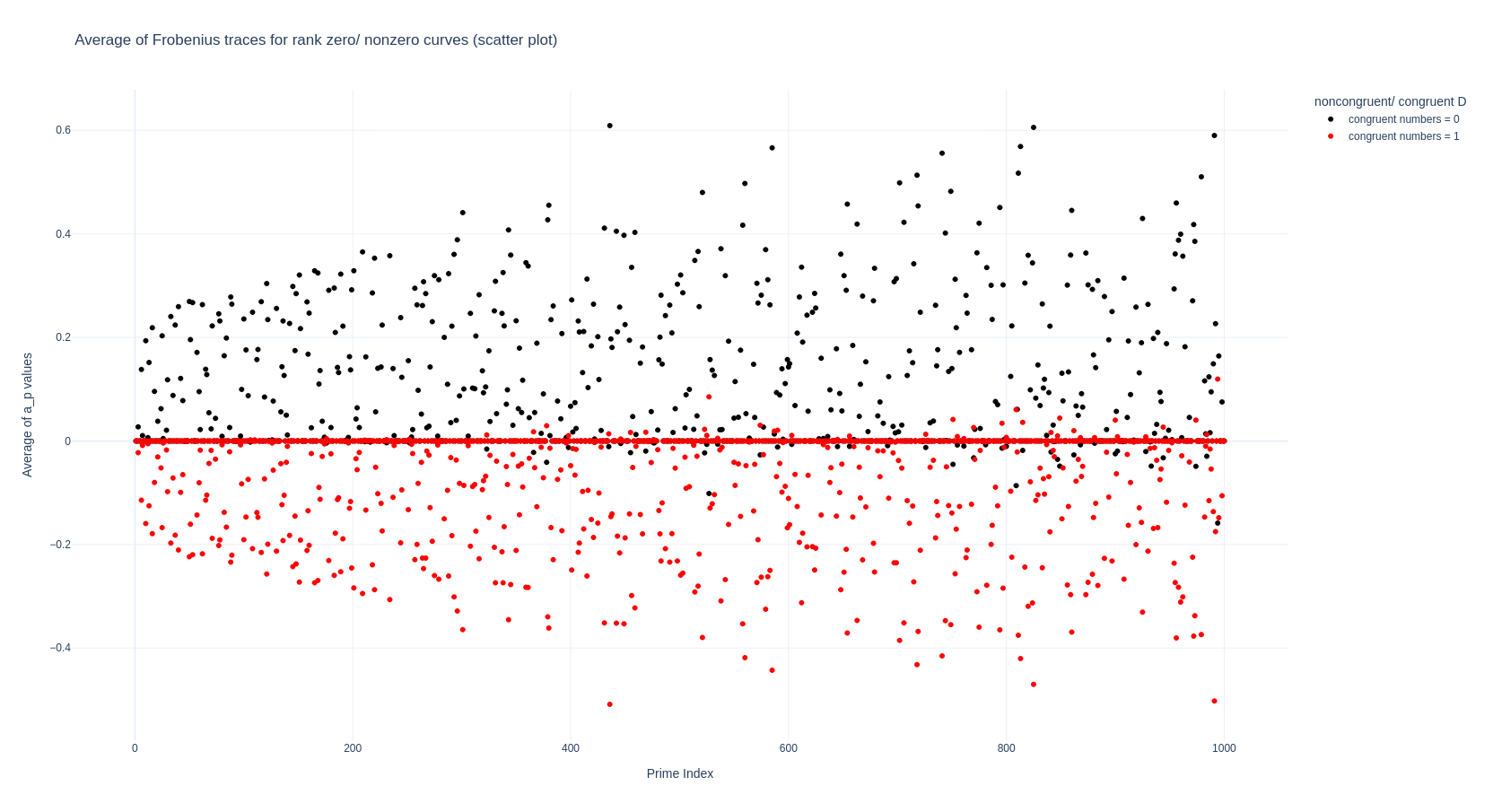}
    \caption{Average of Frobenius traces for rank $0$ and rank nonzero curves}
    \label{traceavgrank}
\end{figure}

As the second experiment, we took averages across residue classes of $D \pmod{8}$.
The scatter plot of these averages can be seen in the left side of \Cref{frob01}. 
Note that, at each prime there are $6$ averages, and they are color coded as given in the figure. 
It is clear that there is no nice \emph{murmurations} pattern as observed in \cite{mumur01}. 
The empirical averages are small, however, they increase as the primes increase. 
If we smoothen this scatter plot using cubic spline technique and use different colors for different residue classes, then we do see some nice oscillations, with varying amplitudes. 
The amplitude variation is present due to the factor of $a_p(E_1)$. 

\begin{figure}[ht]
    \centering
    \begin{tabular}{cc}
        \includegraphics[clip,width=0.45\textwidth]{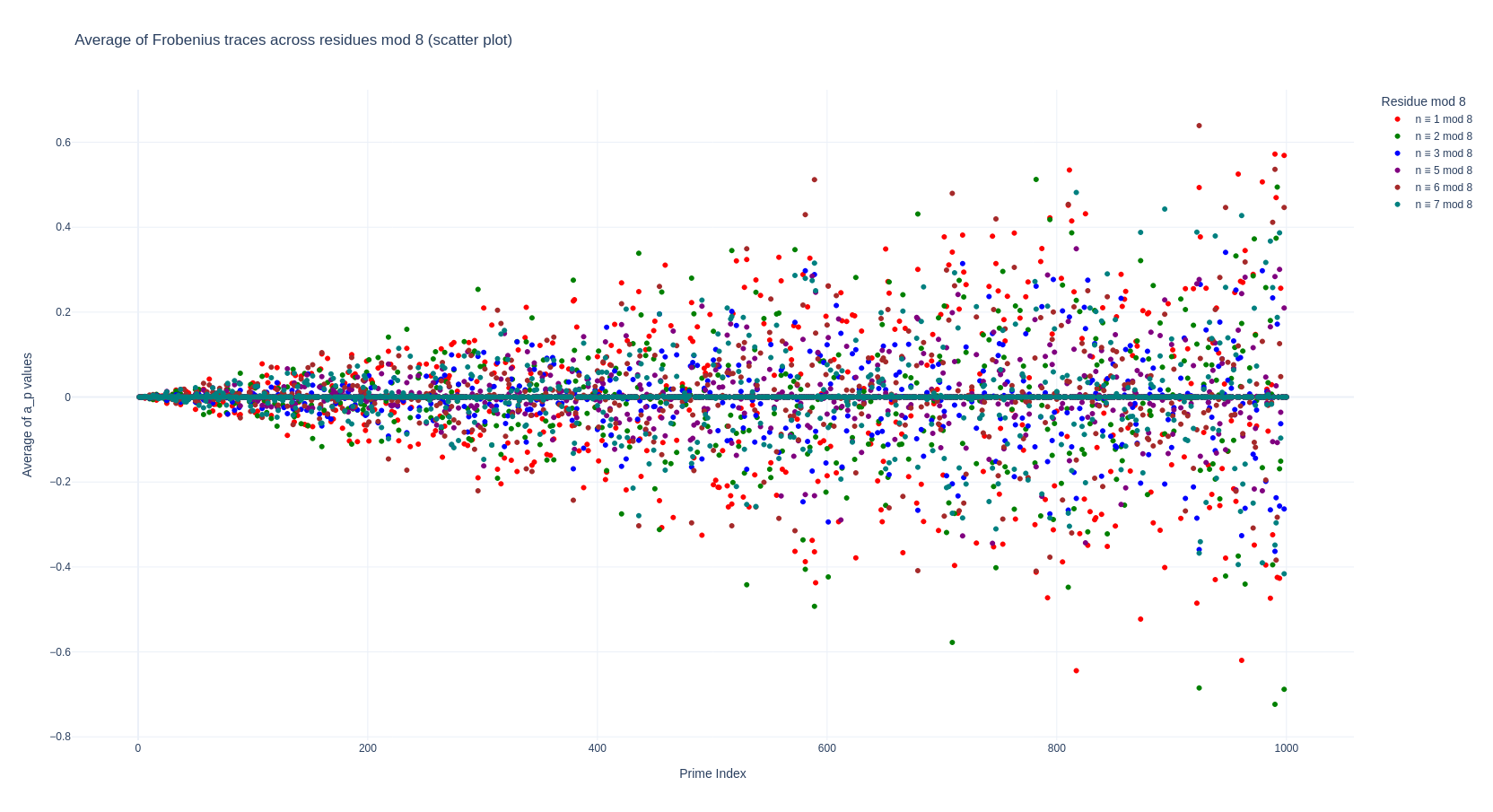} & \includegraphics[clip,width=0.45\textwidth]{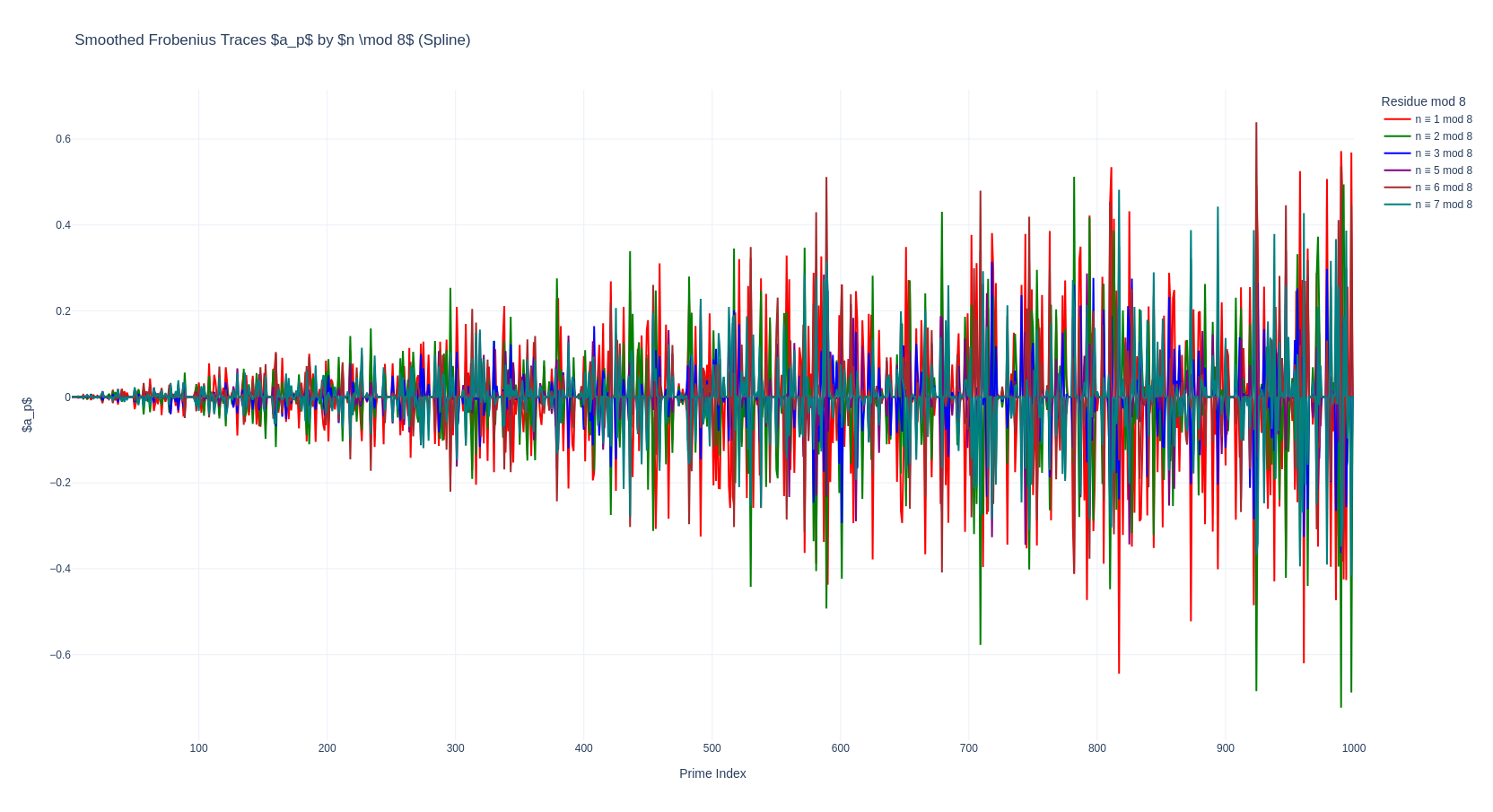}  
    \end{tabular}
    \caption{Averages computed across residues modulo $8$, on the left is the scatter plot and on the right is the smoothed version.}
    \label{frob01}
\end{figure}

We normalize these traces by dividing by $2\sqrt{p}$, hence bringing all the values in the range $[-1, 1]$. 
The scatter plot of normalized averages is in \Cref{frob02}. 
It is clearly seen that the absolute value of the averages is less than $10^{-6}$. 
Note, that these averages are calculated over all the curves in the data. 

\begin{figure}[ht]
    \centering
    \includegraphics[width=0.65\linewidth]{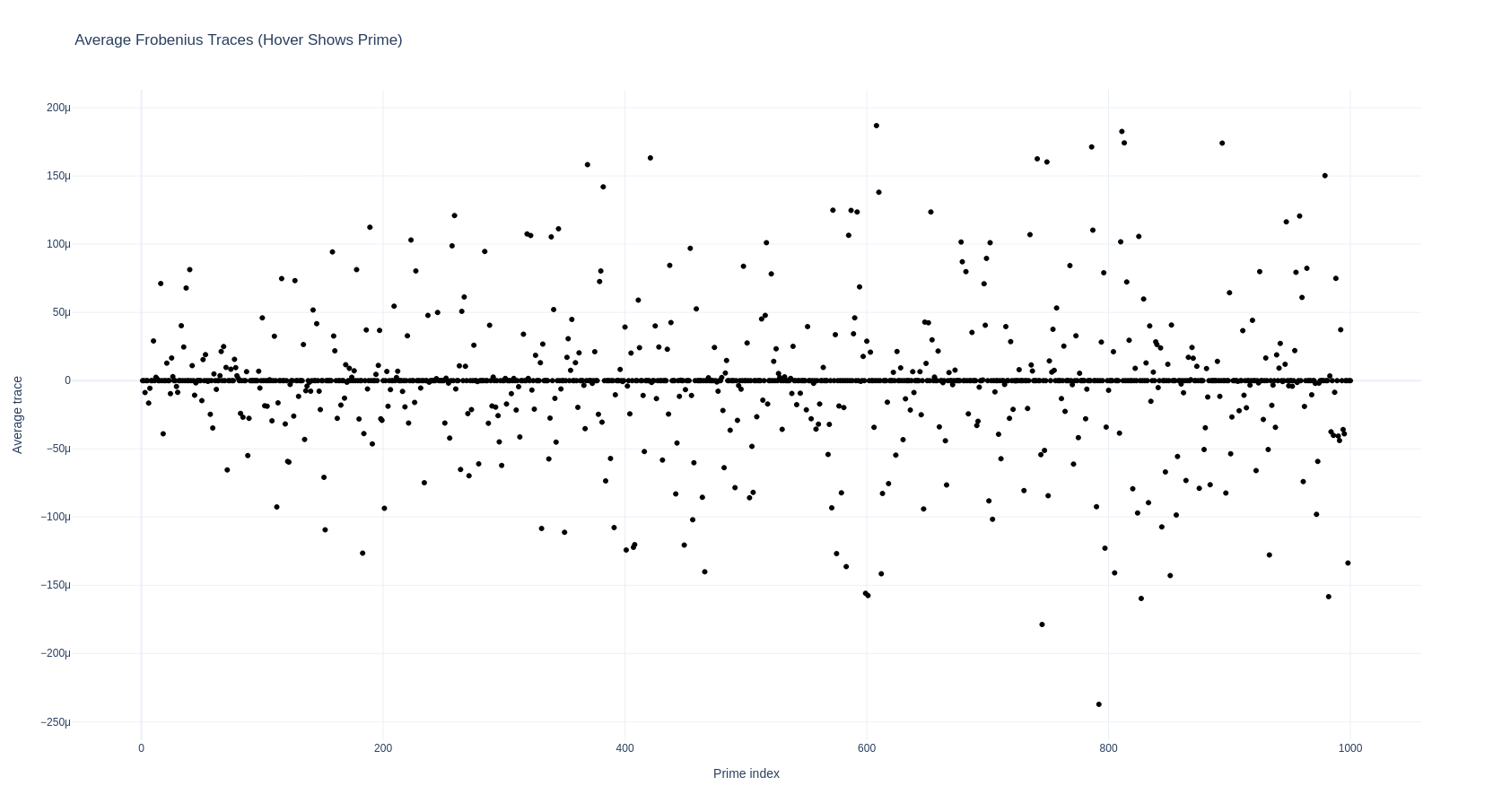}
    \caption{Normalized average over all curves}
    \label{frob02}
\end{figure}

These experimental observations lead to the following result. 
To the best of our knowledge, it has not been mentioned in the literature. 
Arbitrarily choose a prime $p$ and consider the following sequence of averages of the Legendre symbol:
\[C_n(p) := \frac{1}{n}\sum_{k\leq n} \left(\frac{k}{p} \right) \]
\begin{lem}\label{legavg} With the notation as above,
    \[\lim_{n\to\infty} C_n(p) = 0.\]
\end{lem}
\begin{proof}
    Note that, modulo $p$, half the numbers are quadratic residues and the half are quadratic non-residues. Which implies, 
    \begin{align*}
       -\frac{p-1}{2} \leq \sum_{k\leq n} \left(\frac{k}{p} \right) & \leq \frac{p-1}{2} 
    \end{align*}
    Hence, $C_n(p)$ is squeezed between two sequences that tend to zero. 
\end{proof}
As an immediate consequence we have, 
\begin{thm}\label{frobavg}
We have the following limiting behavior of the Frobenius average function:
    
    \[\lim_{X\to\infty} f_{X} (n) = 0 \]
    for every $n\in\mathbb{N}$. 
\end{thm}
\begin{proof}
    The proof is a straightforward application of \Cref{legavg}. 
    \begin{align*}
      \lim_{X\to\infty} f_{X} (n) &= \lim_{X\to\infty} \frac{1}{X} \sum_{D\leq X} a_{p_n}(E_D) \mu(D)^2 \\
      & = \lim_{X\to\infty} \frac{a_{p_n}(E_1)}{X} \sum_{D\leq X}\left( \frac{D}{p_{n}} \right) \mu(D)^2\\
      & = 0.
    \end{align*}
    Where $\mu$ is the M\"obius function. 
\end{proof}
We now demonstrate the above theorem using data.
\Cref{movingavg} depicts the behavior of $f_X(37)$ as $X$ increases from $1$ to $1$ million ($157$ is the $37$th prime). 
\begin{figure}[ht]
    \centering
    \includegraphics[width=0.5\linewidth]{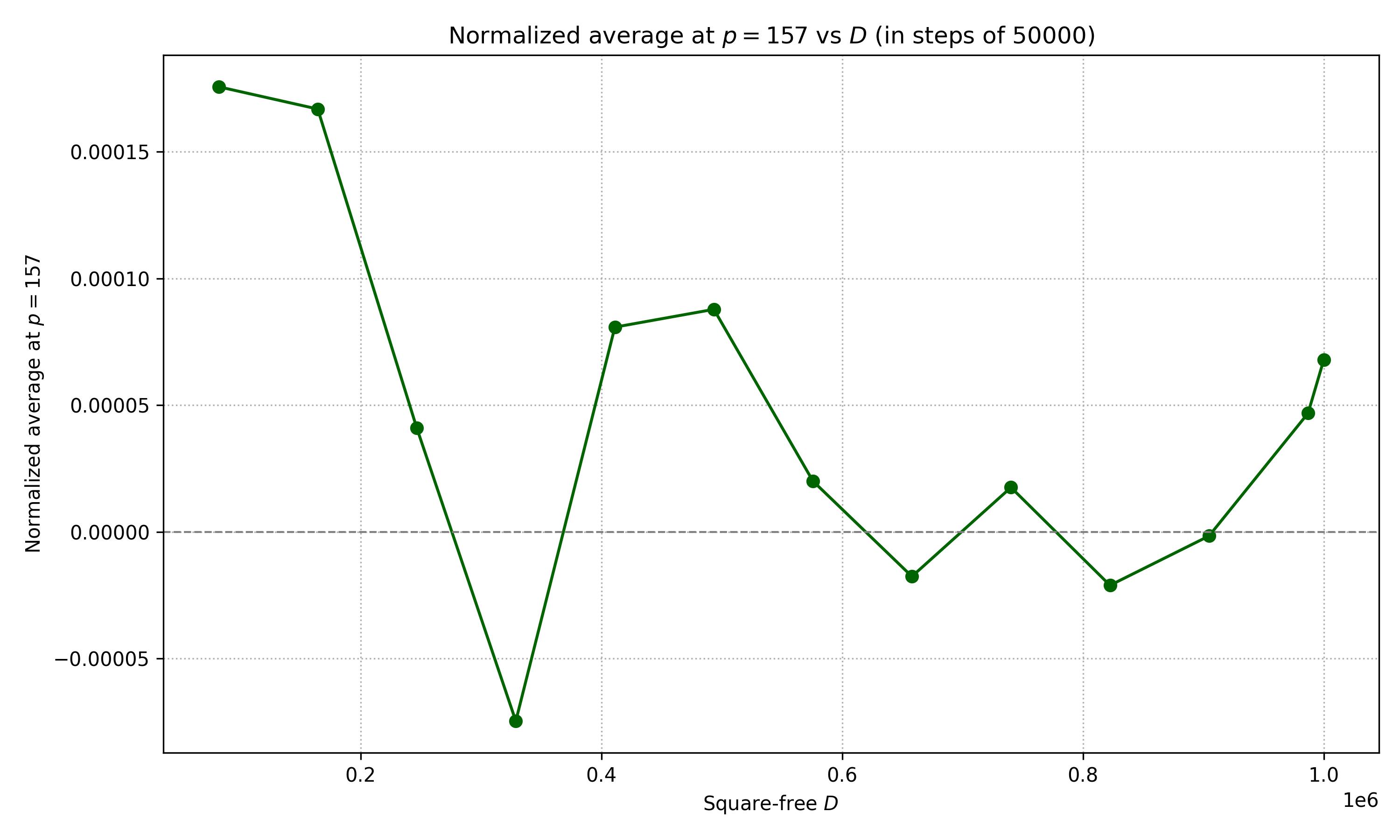}
    \caption{The average trace at $p = 157$ as $X$ increases}
    \label{movingavg}
\end{figure}

Note that, there is nothing special about the congruent number elliptic curves, the above result will be true for any quadratic twist family. 
In fact, there is nothing special about the quadratic twist family. 
We did the same experiment for the cubic twists of $y^2 = x^3 - 1$ and quartic twists of $y^2 = x^3 - 2x$, and discovered the same phenomenon. 
\Cref{cubicavg} shows the average of Frobenius trace of the cubic and the quartic twist families at the prime $p=157$. 
It is clear that as $X$ increases the average oscillates and tends to $0$. 
\begin{figure}[ht]
    \centering
    \begin{tabular}{cc}
    \includegraphics[width=0.45\linewidth]{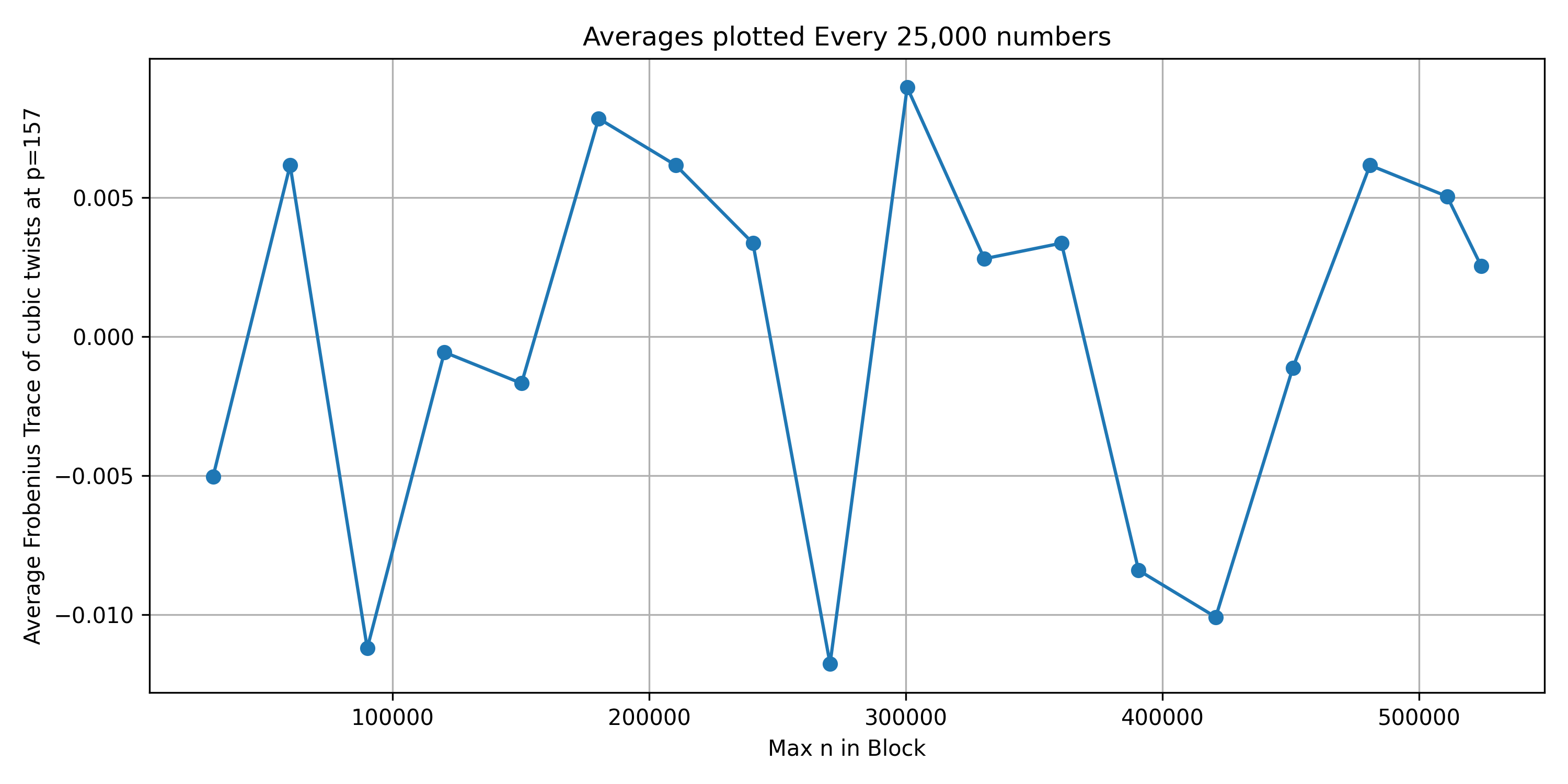} &
    \includegraphics[width=0.45\linewidth]{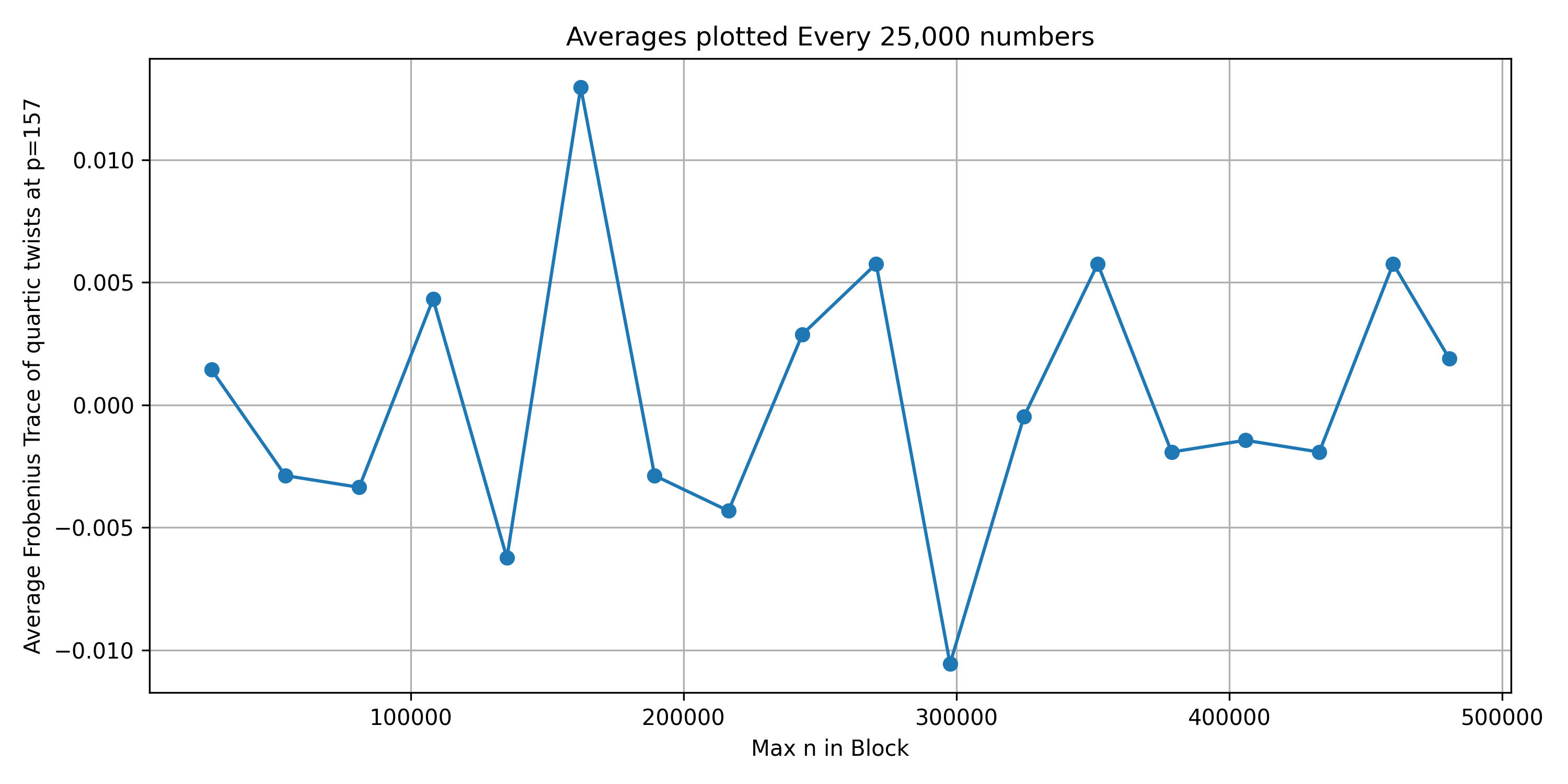}
    \end{tabular}
    \caption{Average value of Frobenius traces over a cubic (left) and a quartic (right) twist family}
    \label{cubicavg}
\end{figure}

These data-driven insights suggest the following:
\begin{conj}\label{openq} 
    Let $\mathcal{E}$ be some twist family (quadratic, cubic, quartic, sextic etc.). 
    If for every $E\in\mathcal{E}$ and prime $p$, the normalized trace $\overline{a}_p(E)$ is seen as a discrete random variable on $\{-1, 0, +1 \}$ then the expected value of $\overline{a}_p()$ zero.
\end{conj}

\section{Machine learning congruent numbers}\label{sec:mlexp}
In this section, we describe our experimental methodology whereby we employ supervised machine learning techniques to ascertain whether a given square-free integer is a congruent number. 
We systematically aggregate similar parameters and employ these parameters as features to train a variety of machine learning models. 
For instance, we utilize $a_p$ values as features and conduct experiments to evaluate how accurately these values can predict whether the corresponding elliptic curve has a non-zero rank, which is equivalent to determine whether the number is congruent. 
The strategy is structured as follows: for each category of features, we construct a balanced dataset where the number of curves with zero rank is equal to those with nonzero rank. 
The selection of curves is carried out uniformly at random. 
Subsequently, $80\%$ of the curves are allocated for training the model, while the remaining $20\%$ are reserved for testing the model. 
For each model, the classification report includes the accuracy, precision, recall, F1 score, and ROC AUC, in addition to the confusion matrix. 
The detailed reports are available in our Github repository. 
The study utilizes the following five machine learning models.
\begin{enumerate}
    \item \textbf{Logistic regression} is a linear model that uses a sigmoid function to map the output of a linear combination of input features to a probability between 0 and 1. 
    This probability is then used as a threshold to classify observations into one of two categories, making it a statistical method for binary classification.
    \item \textbf{Random forest classifier} is an ensemble method that builds many decision trees on different random subsets of the data and combines their predictions. 
    This aggregation (typically by majority vote) reduces overfitting and improves accuracy compared to a single tree. 
    It is robust and interpretable in terms of feature importance.
    \item \textbf{A decision tree algorithm} partitions the data into subsets based on rules learned from the features. 
    It recursively splits the input space using axis-aligned thresholds to predict the class label at each leaf. 
    This method resembles a flowchart of if-then rules optimized for classification.
    \item \textbf{Gradient boosting algorithm} constructs a sequence of decision trees, where each tree attempts to correct the errors made by the previous ones. 
    It minimizes a loss function by fitting new trees to the residuals of earlier predictions. 
    The method is highly flexible and effective for structured data.
    \item \textbf{XGBoost} (Extreme Gradient Boosting) is a scalable and optimized implementation of gradient boosting. 
    It incorporates regularization, efficient tree pruning, and parallel computation, making it particularly powerful for high-dimensional and noisy datasets. 
    XGBoost often achieves state-of-the-art performance in classification tasks.
\end{enumerate}

For a general background on the classification methods and their implementation, see \cite{ESL} and \cite{ISL}. 
The detailed classification reports, including the confusion matrix, are available in our repository. 

\begin{expi}
Our objective here is to ascertain the extent to which arithmetic properties, such as residue classes of a given $D$ modulo $8, 16, 32$ and the number of prime factors of $D$, determine whether it is congruent. 
Conjecturally, all numbers that are $5, 6, 7\pmod{8}$ are congruent, a hypothesis we have empirically validated. 
Consequently, we examine residues modulo $16$ and $32$, with the anticipation that this finer division will facilitate the identification of non-congruent numbers in $D\equiv 1, 2, 3\pmod{8}$.

Our database lists square-free numbers up to $1$ million, a column named \texttt{iscongruent} showing $1$ if $D$ is congruent, otherwise $0$. 
These values are based on the MW rank of $E_D$. 
It includes three more columns for the residue of $D$ modulo $16$, modulo $32$, and the number of prime factors of $D$. 
The dataset contains $140000$ curves, evenly split between congruent and non-congruent numbers.

All the models achieve nearly $96\%$ accuracy, with near perfect precision. 
In fact, the gradient boosting algorithm predicts all the non-congruent numbers and there are no false positives. 
The recall rate of all the models is around $91\%$ and the number of false negatives is nearly same for each model.

Let's examine the distribution of congruent numbers by the number of prime factors in each residue class modulo $16$ (see \Cref{cndist16omega}). We note the following:
\begin{enumerate}
    \item Primes of the form $3\pmod{8}$ are non-congruent; an emprical verification of a known result. 
    \item Semi-primes of the form $10\pmod{16}$ are non-congruent, represented as $2q$ where $q\equiv 5\pmod{8}$ is an odd prime. Specifically, $2$ is a quadratic non-residue modulo $q$.
    \item $D = pq$ is an odd semi-prime such that $D\equiv 3\pmod{8}$. If $p$ is not a quadratic residue modulo $q$, then $D$ is not congruent. Conversely, when $p$ is a quadratic residue modulo $q$, then $90\%$ of the $D$s are not congruent while $10\%$ are congruent. 
    \item The proportion of congruent numbers in each residue class rises with more prime factors.
\end{enumerate}
These observations could be used to fine tune the machine learning models as well as to find better characterization of non-congruent numbers. 
One could also build models that have Legendre symbols as features. 
\begin{figure}[ht]
    \centering
    \includegraphics[width=0.75\linewidth]{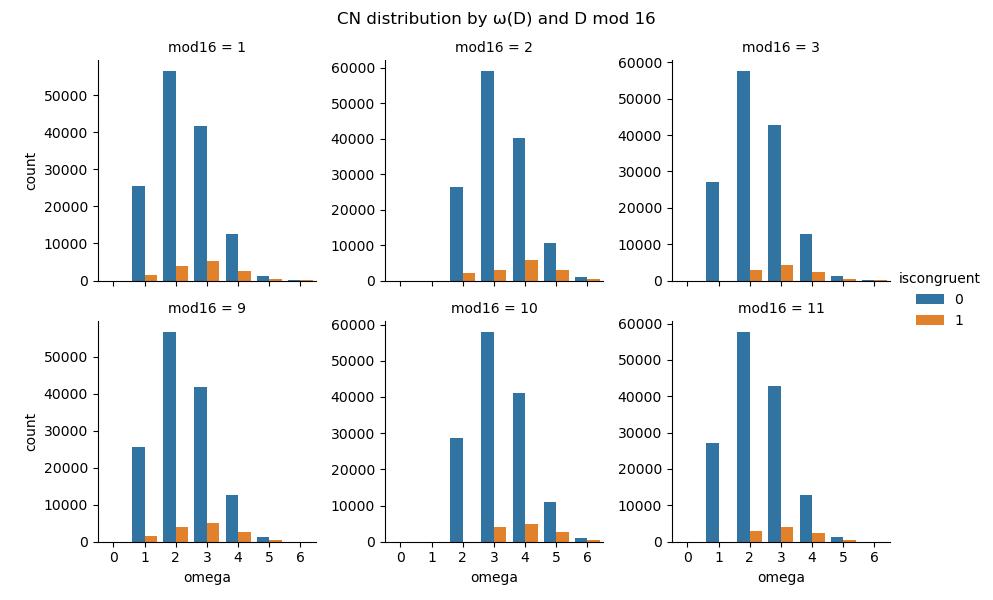}
    \caption{Distribution of congruent numbers: across residue classes modulo $16$ and according to the number of prime factors.}
    \label{cndist16omega}
\end{figure}
\end{expi}

\begin{expi}
    This is the dataset consisting of $1,95,800$ curves corresponding to square-free integers up to $322061$. 
    For each of these curves we calculate its regulator, Tamagawa product, torsion (= $4$), the real period and the special value. 
    We create a balanced dataset containing $87312$ curves corresponding to congruent numbers and the same number of curves corresponding to non-congruent numbers.
    The random seed is fixed for reproducibility. 
    The dataset is then shuffled; $20\%$ of the curves are kept aside as a test set which is not used during training. 
    Almost all the models perform very well. 
    In particular, random forest, gradient boosting and decision tree obtain perfect $100\%$ accuracy. 
    The XGBoost model achieves almost $99.8\%$ accuracy, whereas logistic regression achieves $96.28\%$ accuracy. 
    In both the cases the recall score is relatively low, indicating more presence of false negatives. However, the F1 score is good in both the cases. 
\end{expi}

\begin{expi}
   Selmer theoretic features: for this experiment we consider the $2$-Selmer rank, the $3$-Selmer rank and the $2$-valuation of the modular degree. 
   Recall that the $2$-valuation of the modular degree is the integer $r$ such that the degree is $2^r$ times an odd number. 
   We note that, in our dataset of more than $600,000$ curves all the modular degrees are even and the $2$-valuation is always bigger than the MW rank, which is consistent with a conjecture of Watkins that $2^{\mathrm{rank}(E)}$ divides the modular degree. 

   Next we prune our dataset to $59,110$ curves (corresponding to square-free numbers up to $95,000$) since we have $3$-Selmer ranks calculated for only those many curves. 
   Half of these curves are of rank $0$ and the remaining half have non-zero rank; thus forming our balanced dataset. 
   All the models uniformly predict congruent numbers with an accuracy (and the F1 score) of slightly more than $96\%$. 
   Due to a higher number of false negatives the recall is around $95\%$, indicating the presence of false negatives. 
\end{expi}

\begin{expi}
    Frobenius traces: in this experiment we consider these traces as our features.  
    Our dataset contains all the square-free numbers up to $1$ million, equivalently $607,926$ congruent number curves, as features we have Frobenius traces of each curve for the first $1,000$ primes. 
    A balanced dataset is created by randomly selecting $70,000$ curves corresponding to congruent numbers and $70,000$ that do not correspond to congruent numbers. 

    Across models the accuracy barely crosses the $50\%$ mark. 
    The Recall rate is extremely poor, on average below $49\%$, indicating significant presence of false negatives. 
    Either increasing the sample size or tweaking train/test ratio makes no change. 
    For a congruent number elliptic curve, the Frobenius trace at a prime $p\equiv 3\pmod{8}$ is $0$. 
    There are $505$ such primes in our dataset, hence we trim the dataset by retaining primes of the form $1\pmod{4}$. 
    A marginal increase in the accuracy is observed in case of the Logistic Regression model, however, in case of other models, the results don't change much. 

    If one assumes the BSD conjecture, then the value of the $L$-function of $E_D$ at $1$ should determine whether $D$ is a congruent number. 
    In particular, if $L(E_D, 1) = 0$ then $D$ is a congruent number otherwise not. 
    Recall that $L(E, 1)$ has a Taylor expansion, called the Dirichlet series. The coefficients of this series are implicitly determined by the Frobenius traces. 
    Since all possible Frobenius traces are needed to uniquely determine the $L(E, 1)$, our experiments suggest that more primes are needed in order to improve the accuracy. 
    
\end{expi}
There are many directions after these experiments. 
For example, the interested reader can choose one algorithm and find important features in each of these experiments. 
Then, one can try and train models on these \emph{important} features. 
One could also explore training neural networks for this prediction task.

\section{Concluding Remarks}\label{conclude}
In the final Section we present some empirical data that we hope will lead to new discoveries. 
For example, to date there is no research regarding $3$-Selmer groups similar to that of Heath-Brown's work on $2$-Selmer groups of congruent number curves. 
In our experiments we calculated $3$-Selmer ranks and have detailed the findings here. 
Similarly, there are not many results regarding distribution of MW rank, so we perform similar experiments. 

\subsection{Distribution of \texorpdfstring{$3$}{3}-Selmer groups} 
To the best of our knowledge there has been no study undertaken to understand the distribution of $3$-Selmer ranks of congruent number elliptic curves. 
As described earlier, we calculated $3$-Selmer ranks using the \texttt{ThreeSelmerGroup} function of Magma. 
When compared to $2$-Selmer computations, calculation of $3$-Selmer ranks is considerably slow. 
With increasing conductor for some curves it takes hours to compute the $3$-Selmer group. 
Our current data contains ranks of $80,760$ curves, these correspond to square-free numbers up to $1,30,843$. 

The first striking feature we discovered was that, similar to the $2$-Selmer case, the parity of $3$-Selmer rank depends on the residue of $D$ modulo $8$. 
As seen in \Cref{threeseldist}, in case of $1, 2, 3\pmod{8}$, the ranks are even. 
In particular, in each of these three residue classes almost $66\%$ have ranks $0$, slightly less than $34\%$ have rank $2$ and only a handful are of rank $4$. 
In case of $5, 6, 7\pmod{8}$, the ranks are always odd. 
Almost $97\%$ curves in each residue class are of rank $1$ and the remaining $3\%$ are of rank $3$. 

\begin{figure}[ht]
    \centering
    \includegraphics[width=0.7\linewidth]{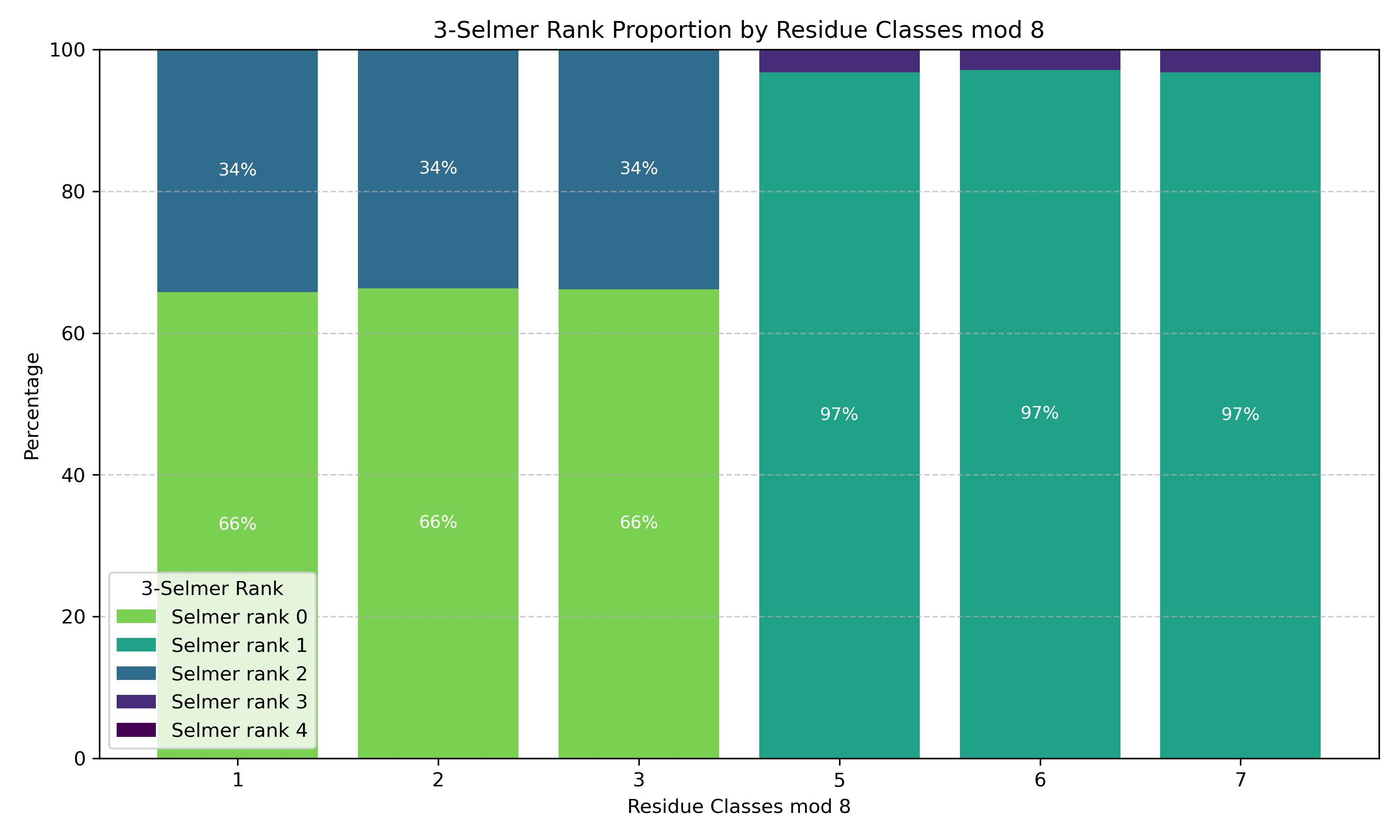}
    \caption{Distribution of $3$-Selmer ranks}
    \label{threeseldist}
\end{figure}

The average size of the $3$-Selmer group, calculated for the entire set of curves is $3.75376$.
Moreover, the average size appears to be constant across all residue classes, see \Cref{s3size}. 
Interestingly, these numbers are pretty close to the conjectured average of $4$, across all elliptic curves.
This behavior is strikingly different from the $2$-Selmer case. 
\begin{table}[ht]
    \centering
    \begin{tabular}{|c|c|} \hline
         \textbf{Residue}&\textbf{Average}  \\ \hline
        $1$ & $3.733086$ \\ \hline
        $2$ & $3.754089$ \\ \hline
        $3$ & $3.772318$ \\ \hline
        $5$ & $3.798515$ \\ \hline
        $6$ & $3.721604$ \\ \hline
        $7$ & $3.742984$ \\ \hline
    \end{tabular}
    \caption{Average size of $3$-Selmer group across residue classes}
    \label{s3size}
\end{table}

The probability mass function of $3$-Selmer ranks is illustrated in 
\Cref{3selpmf}. 
The empirical probability distribution is fairly close to the theoretical one, as conjectured by Poonen-Rains \cite{pr}. 
\begin{figure}[ht]
    \centering
    \includegraphics[width=0.65\linewidth]{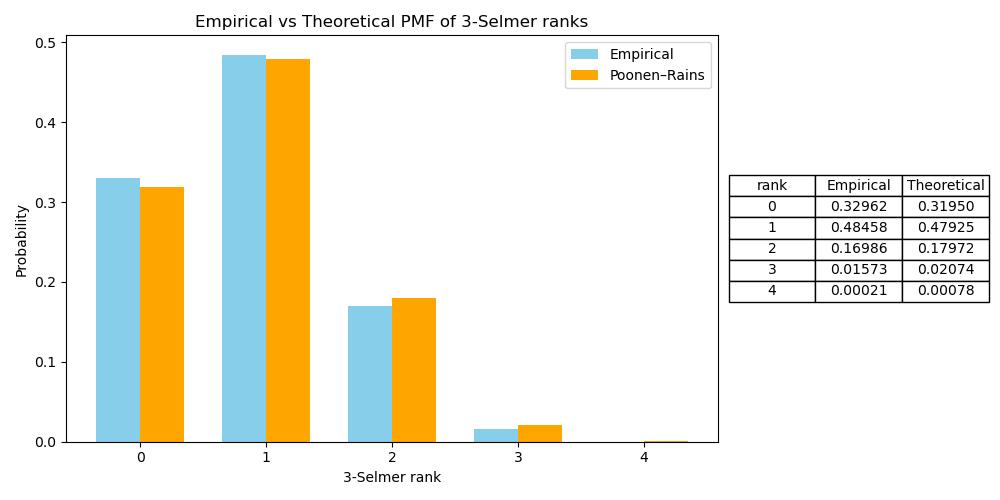}
    \caption{Probability mass function of $3$-Selmer ranks}
    \label{3selpmf}
\end{figure}

In case of moments, the first two moments are close to the theoretical predictions, see, \Cref{3selmpmemnts}. 
From the third moment onward empirical values are much less; probably indicating that the data is too less. 
Recall that the conjectured value the $k$th moment of size of $3$-Selmer groups is $\prod_{i=1}^k(1+3^i)$.

\begin{figure}[ht]
    \centering
    \includegraphics[width=0.55\linewidth]{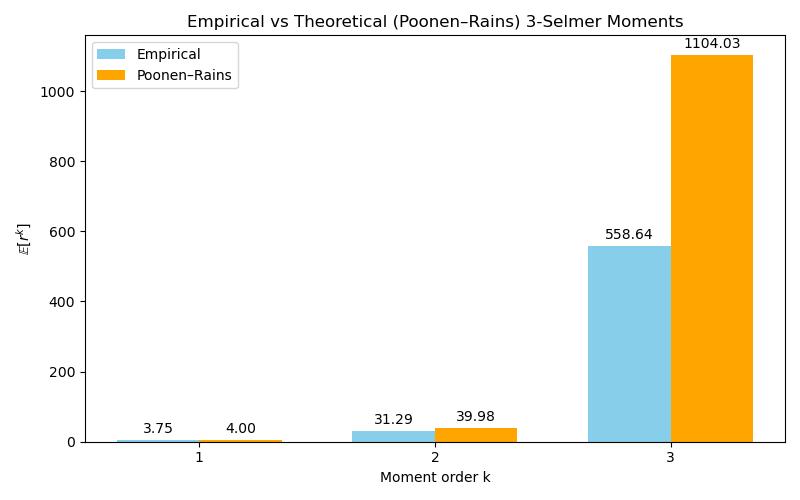}
    \caption{Comparison of moments of $3$-Selmer sizes}
    \label{3selmpmemnts}
\end{figure}

Empirical evidence suggest that there is much scope for theoretical investigations of $3$-Selmer groups of congruent number elliptic curves. 
Similar to the work of Heath-Brown, distribution of $3$-Selmer ranks should be studied in-depth. 
We end this section by exhibiting the computations of $3$-Selmer moments across odd residue classes, analogous to 
\Cref{hb2thm1}. 
Empirical data shows that the moments are almost constant across the four residue classes, again a very different behavior compared to the $2$-Selmer case. 
The details are in \Cref{odd3selmomemnts}.

\begin{figure}[ht]
    \centering
    \includegraphics[width=0.55\linewidth]{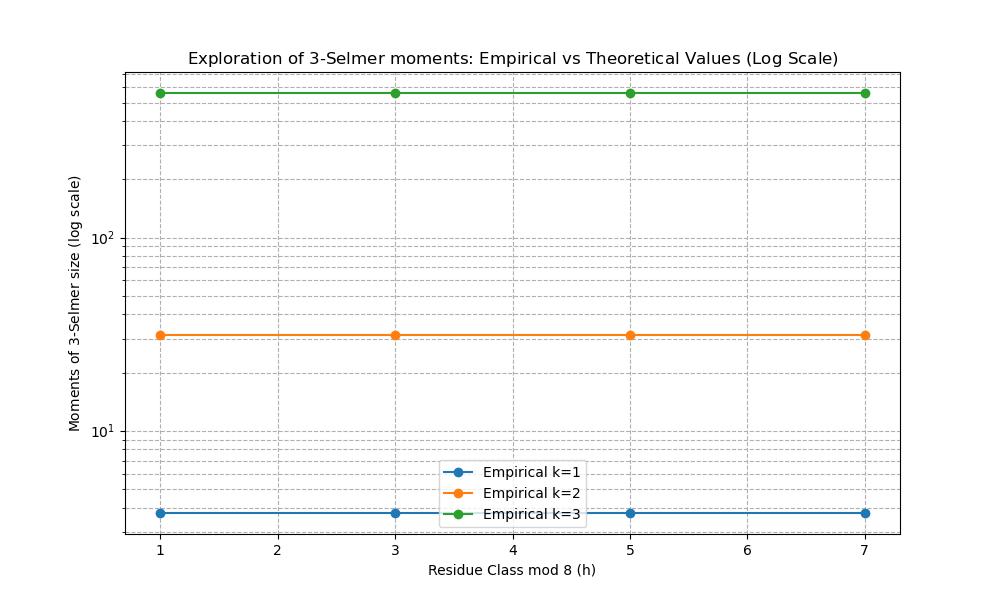}
    \caption{Comparison of moments of $3$-Selmer sizes across odd residues}
    \label{odd3selmomemnts}
\end{figure}

\subsection{The behavior of the MW rank: an empirical perspective}
In general, the distribution of MW ranks is not well understood. 
Since, the Selmer rank is an upper bound for the MW rank, Heath-Brown in his paper, has mentioned some results that provide an upper bound for trailing probabilities and average rank across residue classes modulo $8$. 
In this Section, we present empirical findings and hope that they might lead to improved bounds and a better understanding. 

The database contains MW rank of $N = 1,734,926$ curves corresponding to most square-free numbers up to $3$ million. 
The average rank of these curves is $0.55995$. 
The average rank across each residue class is given in \Cref{rankmean}.
\Cref{rnkpro} shows the stacked bar chart of proportions of these ranks across all the $6$ residue classes. 
It is clear that the parities of the Selmer rank and the MW rank match. 
In particular, for residue classes $5, 6, 7$ almost all curves are of rank $1$ and less than $0.5\%$ have rank $3$. 
Implying that the corresponding numbers are congruent. 
The curves corresponding to the residue class of $3$ have the maximum proportion of rank $0$, more than $91\%$.

\begin{table}[ht]
    \centering
    \begin{tabular}{|c|c|}
    \hline
    \textbf{Residue class} & \textbf{Average rank}\\ \hline
        $1$ & $0.170107$  \\ \hline
        $2$ & $0.169015$ \\ \hline
        $3$ & $0.127843$ \\ \hline
        $5$ & $1.002960$ \\ \hline
        $6$ & $1.005074$ \\ \hline
        $7$ & $1.002631$ \\ \hline
    \end{tabular}
    \caption{Average MW rank across residue classes}
    \label{rankmean}
\end{table}

\begin{figure}[ht]
  \centering
  \includegraphics[width=0.6\textwidth,clip]{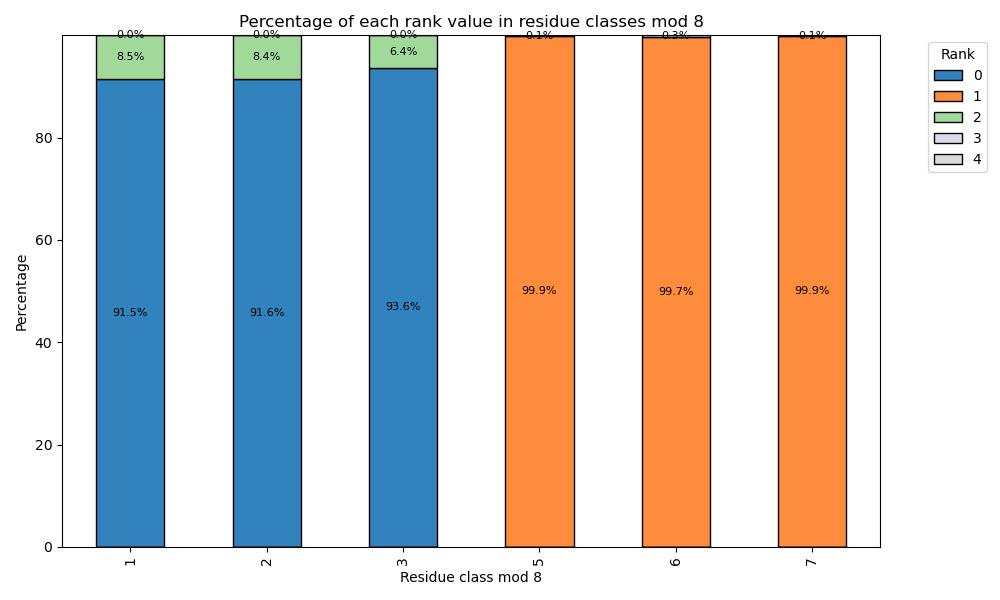}
  \caption{Proportion of MW ranks in each residue class}
  \label{rnkpro}
\end{figure}

We start our analysis with \cite[Corollary 3]{hb}, which says that the rank  $\mathrm{rk}(E_D) = 0$ for at least a third of all $D\equiv 1, 3\pmod{8}$ and it is $1$ for at least five-sixths of all $D\equiv 5,7\pmod{8}$. 
Moreover, the average across odd residue classes is bounded above by $1.5$. 
Empirically, the averages are fairly small as seen in \Cref{rankmean}. 
On the other hand, proportion of rank $0$ curves is $0.9092$ and the proportion of rank $1$ curves is $0.9946$, much bigger than the theoretical prediction. 

We now examine \cite[Corollary 3]{hb1}, which gives an upper bound for trailing probabilities across odd residues, i.e., the proportion of curves with rank at least $r$. 
The upper bound is the same given above in \Cref{trailingprob}. 
A plot of our empirical observations is in \Cref{mwrtrail}, as seen the theoretical upper bounds are not at all tight; there is a possibility for improvement.

\begin{figure}[ht]
    \centering
    \includegraphics[width=0.5\linewidth]{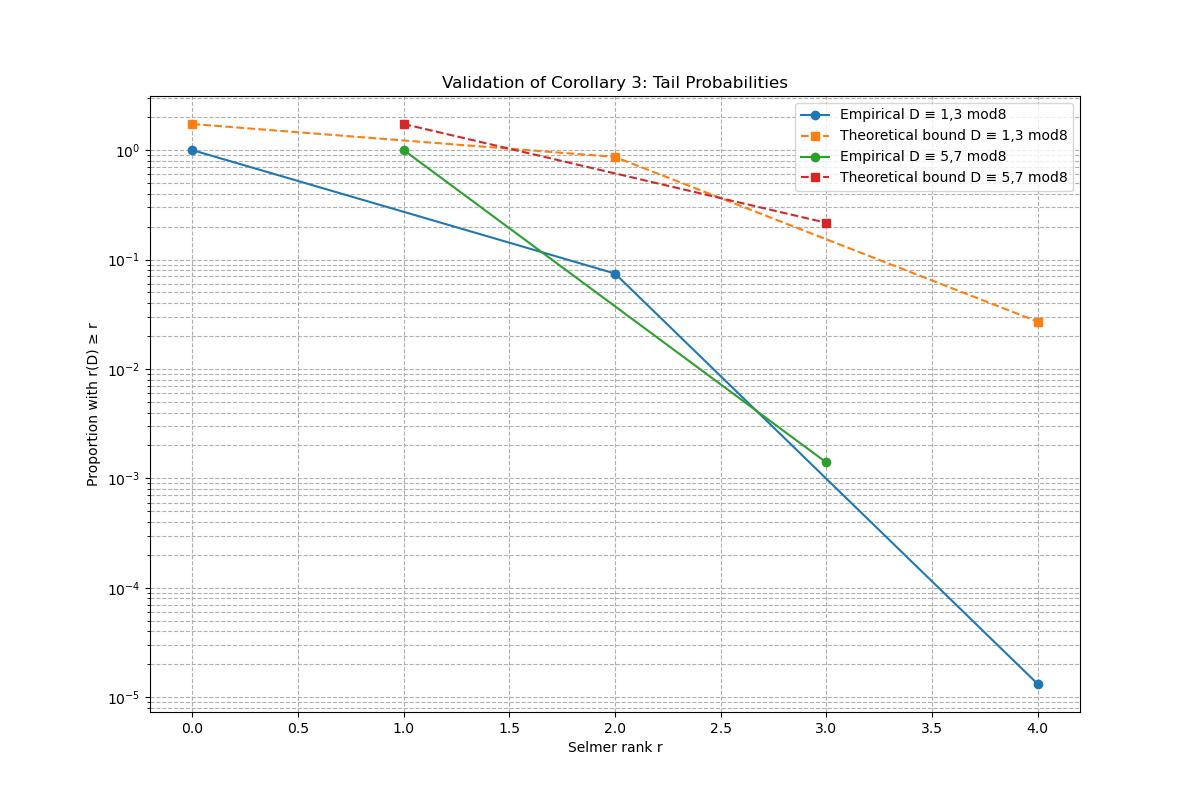}
    \caption{Trailing probabilities of the MW rank}
    \label{mwrtrail}
\end{figure}

We now look at the average rank across two groups of odd residue classes, i.e., \cite[Corollary 4]{hb1}. 
The \Cref{avgmwrank} shows the observed findings, it is clear that the upper bounds predicted by the Corollary are pretty far off. 
Again, there is room for improvement. 

\begin{table}[ht]
    \centering
    \begin{tabular}{|c|c|c|c|}
    \hline
    \textbf{Residue} & \textbf{Empirical} & \textbf{Theoretical} & \textbf{Count}\\ \hline
        $1$ & $0.1701$ & $1.2039$ & $301962$ \\ \hline
        $3$ & $0.1278$ & $1.2309$ & $303153$ \\ \hline
        $5$ & $1.0030$ & $1.3250$ & $275026$ \\ \hline
        $7$ & $1.0026$ & $1.3250$ & $275936$ \\ \hline
    \end{tabular}
    \caption{Average MW rank across odd residue classes}
    \label{avgmwrank}
\end{table}

Finally, we look at the empirical distribution of MW rank across all the curves in our database. 
Though the Goldfeld's conjecture is stated for analytic ranks, we will attempt to verify it for MW ranks instead. 
The first reason is the BSD conjectures, which implies that these two ranks are the same. 
The second reason is that in our experiments with analytic ranks we could not observe the $50/50$ split empirically, mainly due to small data. 
Recall that, Goldfeld's conjecture is an asymptotic statement: the $50/50$ balance is expected in the limiting case, not necessarily for the first few hundred thousand curves.
Hence, we consider $N = 1,734,926$ curves whose MW ranks are computed. 

We note that, random sampling (drawing indices uniformly at random from a large range) might simulate the behavior in that range and break ordering effects.
In particular, Bernoulli sampling method is known for removing ordering biases and transient initial behavior.
Moreover, repeated random samples let us estimate the distribution of proportion of rank $0$ curves, not just a single point estimate. 
If the sampling distribution centers near $0.5$ and narrows with larger sample size, that supports convergence.
Unlike taking the first $k$ rows (by increasing $D$), 

We choose each curve independently with probability $p \in\{\frac{1}{100} ,\frac{1}{200}, \frac{1}{500}\}$. 
We then perform repeated random samples of size $Np$ drawn uniformly from all square-free, $n$ less than $3$ million, where $N = 1,734,926$.
For each sample, we compute the proportion of rank $0$ curves and perform a statistical test of how far are the two distributions.
We repeat this process many times to get the sampling distribution.
From \Cref{bernoulli} it is clear that For $p=\frac{1}{500}$, the sample size is roughly $3500$ and after performing trials for $2,500$ times we see that the maximum proportion of rank $0$ curves crosses $0.51$, otherwise the average is $0.49$. 
Suggesting that we need much more data to empirically observe the true distribution of ranks. 
\begin{figure}[ht]
    \centering
    \includegraphics[width=0.65\linewidth]{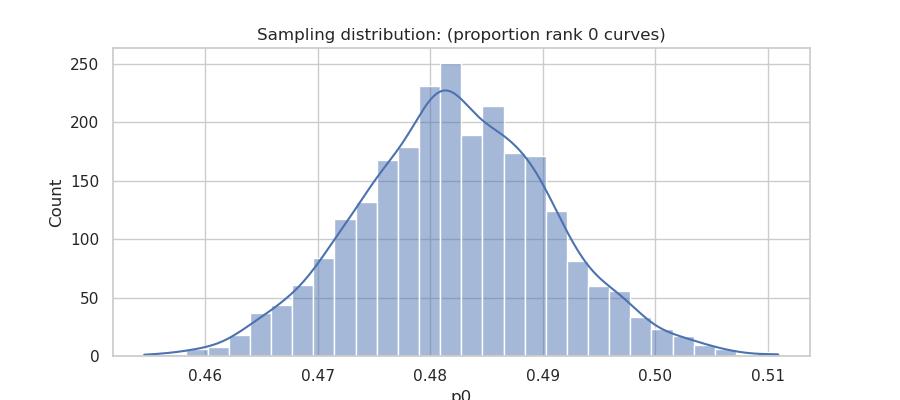}
    \caption{Proportions of rank $0$ curves during Bernoulli trials}
    \label{bernoulli}
\end{figure}

\subsection{The shape of the BSD space}
Principal Component Analysis (PCA) is a dimensionality reduction technique that identifies orthogonal directions of maximal variance in high-dimensional data.
This facilitates both visualization in two or three dimensions and improved clustering of the data with similar intrinsic behavior. 
It provides a geometric lens to explore hidden correlations among (conjecturally) significant quantities.
PCA was used in \cite[Section 4.2]{mumur01} on the Frobenius traces of (randomly selected) $36,000$ elliptic curves and a clear separation according to ranks was observed. 
PCA was also used in \cite[Section 5.3]{shaorder} on the BSD parameters data of almost $3$ million elliptic curves in the LMFDB database whose analytic Sha size is at most $100$. 
In the visualization two distinct arms were observed, but no explanation could be found for this particular shape. 
We repeat this experiment for congruent number elliptic curves. 
We choose five features, namely, torsion, regulator, special value, real period and the Tamagawa product of $2,93,100$ congruent number curves corresponding to square-free numbers corresponding to $4,82,131$. 
We apply PCA to this data; the projection of this $5$-dimensional data on $\mathbb{R}^2$ can be seen in \Cref{shapca}.
The plot is fairly similar to the one obtained in \cite[Figure 5.4]{shaorder}; there are two distinct arms, one of them slimmer; points corresponding to high Sha orders form a dense core. 

\begin{figure}[ht]
    \centering
    \includegraphics[width=0.45\linewidth]{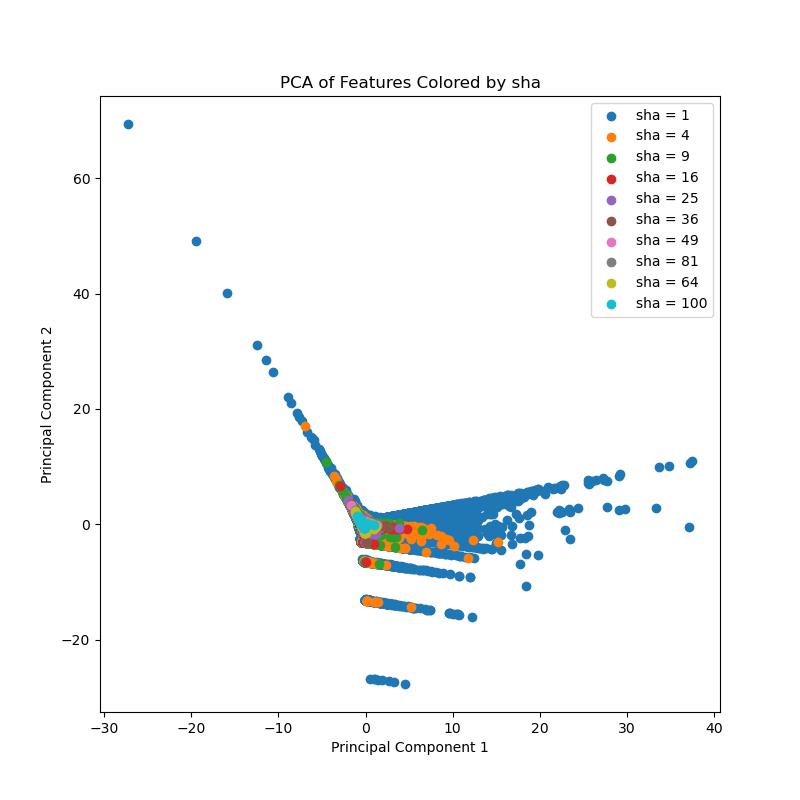}
    \caption{PCA on the BSD parameters with Sha order at most $100$}
    \label{shapca}
\end{figure}

First, we use persistent homology to understand the $1$-skeleton of the space from which these points are (supposedly) sampled from. 
Since, persistent homology calculations are computationally intensive, we use a subset of data. 
The persistent diagram in \Cref{bsd_pd} shows the existence of at least $3$ distinct connected components. 
Since, all the $1$-dimensional persistent homology classes are near the diagonal the BSD-space has no $1$-dimensional homology. 
\begin{figure}[ht]
    \centering
    \includegraphics[width=0.45\linewidth]{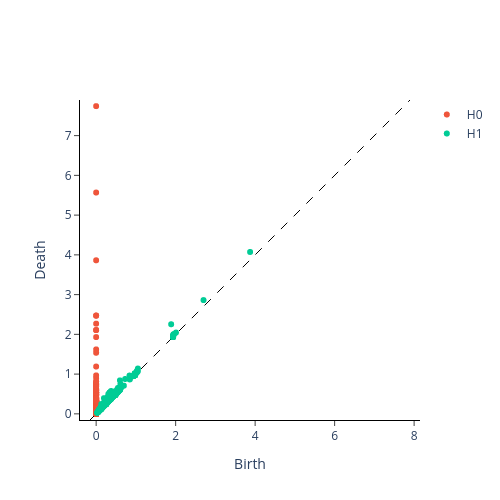}
    \caption{Persistent homology of the space of BSD parameters}
    \label{bsd_pd}
\end{figure}

We then use the Kepler Mapper \cite{keplermapper} library in Python to try and understand this shape.
Recall that, the Mapper algorithm is a popular tool from topological data analysis, which, at its core is based on Morse theory. 
It makes use of dimensionality reduction techniques like PCA as \emph{Morse functions} to describe the shape of the parameter space in terms of a simplicial complex. 
In our case, the Mapper algorithm yields $4$ connected components, which we tried visualizing by different color schemes based on MW rank, Selmer rank, \emph{iscongruent} value and residue class modulo $8$. 
However, we did not find any obvious explanation that may reveal a relationship between these parameters previously unknown. 
The code and the Mapper graphs are available in our repository, interested reader is encouraged to do more experiments in order understand the shape of the BSD space!

\bibliography{references}
\bibliographystyle{plain} 

\end{document}